\documentclass[aos]{imsart}

\RequirePackage{amsthm,amsmath,amsfonts,amssymb}
\RequirePackage[authoryear]{natbib}
\RequirePackage{graphicx}
\usepackage{color}
\usepackage{mathrsfs}
\usepackage{mathabx}
\usepackage{dsfont}
\usepackage{tikz}
\usetikzlibrary{patterns,arrows,decorations.pathreplacing}
\usepackage{subfig}
\graphicspath{{./figure//}}
\usepackage{booktabs, multirow} 
\usepackage{soul}
\usepackage{xcolor,colortbl} 
\usepackage{changepage,threeparttable} 

\startlocaldefs
\theoremstyle{plain}
\newtheorem{theorem}{Theorem}
\newtheorem{corollary}{Corollary}
\newtheorem{proposition}{Proposition}
\newtheorem{lemma}{Lemma}
\theoremstyle{remark}

\numberwithin{remark}{section}

\newcommand{\lmatrix}[1]{\left[\begin{array}{#1}}
\newcommand{\rmatrix}{\end{array}\right]}
\newcommand{\half}{\mbox{\tiny$\frac{1}{2}$}}

\newcommand{\E}{\mathbb{E}}
\newcommand{\PP}{\mathsf{P}}
\newcommand{\EP}{\mathbb{E}_{\mathsf{P}}}

\endlocaldefs

\begin{document}

\begin{frontmatter}
\title{Sequential Scoring Rule Evaluation for Forecast Method Selection}
\runtitle{Sequential Forecast Method Selection}

\begin{aug}
\author[A]{\fnms{D. T.}~\snm{Frazier}\ead[label=e1]{david.frazier@monash.edu}}
\author[B]{\fnms{D. S.}~\snm{Poskitt}\ead[label=e2]{Donald.Poskitt@monash.edu}},

\address[A]{Department of Econometrics and Business Statistics, Monash University\printead[presep={,\ }]{e1}}
\address[B]{Department of Econometrics and Business Statistics, Monash University\printead[presep={,\ }]{e2}}
\end{aug}

\begin{abstract}
This paper shows that sequential statistical analysis techniques can be generalised to the problem of selecting between alternative forecasting methods using scoring rules.  A return to basic principles is necessary in order to show that ideas and concepts from sequential statistical methods can be adapted and applied to sequential scoring rule evaluation (SSRE). One key technical contribution of this paper is the development of a large deviations type result for SSRE schemes using a change of measure that parallels a traditional exponential tilting form. Further, we also show that SSRE will terminate in finite time with probability one, and that the moments of the SSRE stopping time exist. A second key contribution is to show that the exponential tilting form underlying our large deviations result allows us to cast SSRE within the framework of generalised e-values. Relying on this formulation, we devise sequential testing approaches that are both powerful and maintain control on error probabilities underlying the analysis. Through several simulated examples, we demonstrate that our e-values based SSRE approach delivers reliable results that are more powerful than more commonly applied testing methods precisely in the situations where these commonly applied methods can be expected to fail. 
\end{abstract}




\end{frontmatter}

\section{Introduction}

\subsection{Motivation}

In time-series forecasting a basic task that befalls the practitioner is to select from a set of competing alternatives a forecasting system that is likely to deliver the best forecasting
performance. This is commonly done by examining so called out-of-sample loss, defined as the expected value of a measures of the
discrepancy between the predictions and the observed values (the mean squared error for example). The empirical forecaster is therefore concerned with assessing expected performance on as yet unseen data. Researchers also use out-of-sample loss to assess whether a proposed forecasting system outperforms an already established benchmark. Out-of-sample loss is by definition unknown, however, and needs therefore to be estimated.
This is typically done by excluding parts of the observed time-series from the estimation and mimicking the actual process of out-of-sample forecasting by performing
a sequence of predictions for these observations instead. An estimate of the out-of-sample loss is then obtained by
averaging the loss incurred across the individual predictions, \textit{\'{a} la} cross validation. 

Suppose, for example, that the forecaster is asked to provide a probability forecast $p$ for an event $E$ associated with a future value of a random variable of interest $Y$. Key requirements for such a forecast are calibration, meaning that events with a predicted probability of $p$ should occur with a relative frequency of $p$, and sharpness, which requires the forecast probabilities to be informative, i.e. close to 0 or 1. These properties can be assessed using a proper scoring rule \citep{Gneiting2007a}, which in the case of a probability forecast coincides with a consistent scoring function for the mean \citep{Gneiting2011a}. A scoring rule is a function $S(p; y)$ that maps a forecast probability $p$ and an observation $y$ to a numerical score that measures the discrepancy of the forecast, with smaller scores indicating a superior forecast.

To compare two forecasting procedures, $p_t(\tau)$ and $q_t(\tau)$ say, at horizon $\tau\ge1$, the forecaster might collect a sample of $T+\tau$ values of $Y$ and partition the sample into P segments containing training data, in-sample observations used for fitting $p_t(\tau)$ and $q_t(\tau)$, and test data, pseudo out-of-sample observations used for forecast
evaluation. Given $p_t(\tau)$ and $q_t(\tau)$, and the associated observation $y_{t+h}$ for $t = R+1,\ldots,T$, where $R=T+\tau-P$, the average score
difference can be computed as
$$
\Delta_P(p_t(\tau),q_t(\tau))=\frac{1}{P}\sum_{t=R+1}^T[S(p_t(\tau); y_{t+\tau})-S(q_t(\tau); y_{t+\tau})]\,,
$$
and the standardized statistic $\sqrt{P}\Delta_P(p_t(\tau),q_t(\tau))/\widehat{\sigma}$,
where $\widehat{\sigma}^2$ is a consistent estimator of the asymptotic variance of $\Delta_P(p_t(\tau),q_t(\tau))$, is often used to assess the relative accuracy of the forecasts by testing if the average score
difference differs significantly from zero.

A number of tests of the type briefly outlined above are available in the literature, prominent examples being \citet{Diebold1995}, \cite{West1996}, \cite{White2000} and \cite{Hansen2005}, see also the review in \citet{tashman:2000}. Further examples are the martingale-based approaches in \citet{Giacomini2006}, \citet{lai:gross:shen:2011}, and \citet{yen:yen:2021}. A common difficulty with such pseudo out-of-sample evaluation schemes is the resolution of the inevitable trade-off between the size of the data-sets designated as in-sample training data and pseudo out-of-sample test data. The larger is the former the more accurate is the estimation, but the loss is evaluated only on the observations reserved as test data and the smaller is the latter the less precise is the evaluation of the loss. Moreover, the tests mentioned above are all only asymptotically valid. Scarcity of pseudo out-of-sample observations will therefore call into question the accuracy of any asymptotic approximations. Differences in practical implementation can also have a dramatic impact on test performance (see for example \citet{lazuras:lewis:stock:watson:2018} and \citet{frazier:covey:maerin:poskitt:2023}. Furthermore, the tests mentioned above are couched in terms of testing a null hypothesis of \textit{no inferior forecast performance}, so they are more suited to making a comparison of a new method to an established benchmark or state-of-the-art method rather than choosing between two competing specifications. See \citet{Diebold2015} for further discussion of the pros and cons of using pseudo out-of-sample tests. Sequential methods address these issues and (obviously) are ideally suited to the common scenario where data arrives sequentially and the forecaster is required to evaluate forecasts based on a small recently observed set of data.   The contribution of this paper is to show that the apparatus of sequential statistical methods can be adapted to the analysis of scoring rules, thereby offering the forecaster an alternative methodological approach to forecast method selection.  To the best of the authors' knowledge, the extant literature does not contain a discussion that pursues the particular avenue of research that we consider here.

\subsection{Background}

Sequential statistical methods gained prominence in the 1940's during the second world war, leading to the development of a body of theory around the `sequential probability ratio test' (SPRT) due to \citet{wald:1947}. A summary of related work is also presented in \citet{barnard:1946}. Whereas Wald's theory was mainly devoted to the specification of sampling schemes satisfying given requirements, Barnard's work addressed the converse problem of examining the statistical properties of a given sequential scheme. In the ensuing decades considerable research has been conducted on sequential methods, in particular, the martingale theory in \citet{doob:1953} was used to extend Wald's results, including Wald's fundamental identity, to stochastic processes in continuous and discrete time \citep{dvoretzk:kiefe:wolforwitz:1953,robbins:samuel:1966,brown:1969}, and stationary processes with independent increments \citep{hall:1970}. Text book accounts from a decision theoretic viewpoint can be found in \citet{Ferguson1967} and \citet{degroot:1970}. For a review of the theory and application of sequential methods see \citet{wetherill:1975}. The scenario envisaged in this literature is one where a practitioner has available a data series and has to decide which of two hypothesised data generating processes (DGPs) has produced the observed realization. The decision is based on a boundary crossing rule applied to the `probability ratio' (i.e. the likelihood ratio) given by the probability distributions specified for the two DGPs.

The sequential scoring rule evaluation (SSRE) scheme developed in this paper takes a similar perspective to that just described both in terms of the background problem formulation and the use of a boundary crossing stopping rule. However, a fundamental difference between SSRE and SPRT is that SSRE employs a stopping rule that involves the use of a ratio of scoring rules and not a likelihood ratio. Consequently, standard analysis techniques, such as large deviations theory for independent and identically distributed random variables, or martingale theory, can no longer be relied upon.  The main challenges in analysing SSRE lie in the fact that it is not defined in terms of a likelihood ratio, and a return to basic principles is therefore necessary in order to show that ideas and concepts from sequential methods can be adapted and applied to SSRE. The technical contribution of this paper is the establishment of probabilistic results for SSRE schemes using a change of measure that parallels a traditional exponential tilting form. This allows us to provide descriptions of error probabilities in terms of the levels set for the stopping rule. The inferential characteristics of SSRE can thus be specified in advance, and as information from additional time points is accrued an assessment of relative forecast performance can be made that is valid in finite samples. We also show that SSRE will terminate in finite time with probability one, and that the moments of the SSRE stopping time exist. Thus, SSRE provides the forecaster with an assessment tool that allows the forecaster to stop (or continue) forecast evaluation upon seeing sufficient (or not enough) evidence and, more significantly, SSRE is geared towards the task of forecast method selection rather than model testing.

To implement the proposed SSRE testing approach, we show a novel link between SSRE schemes and e-variables. In particular, we show that our proposed SSRE scheme is a universal generalized e-value (GUe-value) under our null hypothesis; GUe-values were proposed and studied in \citet{dey2024anytime} for parameter inference and by \citet{dey2024multiple} for multiple testing problems. As the name suggests, GUe-values are related to both e-values and universal inference. E-variables have recently emerged as a method for conducting multiple hypothesis testing that do not suffer from the same issues as related p-value-based methods; see \citet{vovk:wang:2021}, \citet{grunwald:heide:koolen:2024}, \citet{shafer:2021} (who refers to e-variables as betting scores), and  \citet{wang:ramdas:2022} for full definitions. On the other hand, universal inference methods were  proposed by \citet{wasserman2020universal} as a method to derive calibrated confidence intervals in well-specified models without any of the usual regularity conditions. While exceedingly useful, the method looses correct calibration in misspecified models. One can show that the correct calibration of the confidence sets in universal inference is a consequence of the fact that the criterion that is inverted to form the confidence set in universal inference is an e-process under the assumed model.

Following \citet{ Giacomini2006} we define a forecasting method as being comprised of a forecasting model together with associated techniques that must be specified by the forecaster at the time the forecast is made, such as what estimation procedure to employ and what data to use. In what follows we allow for the use of different estimation procedures in the construction of the forecasts, including parametric and nonparametric methods, as well as limited memory estimators, such as recursive estimators of the exponential smoothing type, fixed width rolling window estimators, or geometrically weighted expanding window estimators that discount less recent observations. These choices can have a significant affect on future forecast performance over and above those due to the specification of the forecasting model. In the derivation of our results we impose minimal conditions on the DGP and the forecast model and method. This accommodates various types of misspecification and uncertainty resulting from the choices made by the forecaster.

\subsection{Organization of the Paper}
Let $Y_1,\dots,Y_{T},\dots$ denote a sequence generated by a stochastic process defined on a probability space $(\mathsf{Y},\mathscr{B}, \mathcal{P})$. We suppose that the goal of the forecaster is to predict some feature of $F_{Y_{T+h}}$, the distribution of $Y_{T+h}$ at a forecast horizon of $p\geq 1$. We are interested therefore in predictions of a functional $\phi[F_{Y_{T+h}}]$ made at time $T$.  We assume that competing forecasts and the observations $Y_t$ are random processes adapted to a filtration $\mathscr{B}_t$, $t \in \mathbb{N}$, and the measure $\mathcal{P}$ describes the joint dynamics of the forecasts and the observations through its characterisation of the distribution of every finite set $\{Y_1,\ldots,Y_n\}$, which we denote by $\mathsf{P}$, $n=1,2,\ldots$ \citep[see][\S\, 48\,\&\,49]{halmos:1950}. Since in general the true measure $\mathcal{P}$, and hence $\mathsf{P}$, is unknown there is no `correct' way to predict features of interest. We assume therefore that the forecaster considers a class of forecast methods $\mathcal{Q}$ that are identified by their distribution functions $Q\in\mathcal{Q}$, and are constructed using data from realizations $y_1,\dots,y_n$, $n=1,2,\ldots$, of the stochastic process.

In the following section of the paper we outline the class of models and methods under consideration, and provide a brief outline of the scoring rules and sequential estimation techniques that underly SSRE. In Section \ref{SSRE} we present some basic concepts and methods of SSRE and relate these to Wald's SPRT. Section \ref{ssree} establishes the relationship of SSRE to e-variables and generalized e-variables. Section \ref{sec:implement} discusses implementation of the proposed SSRE, and compares this approach against the most common approach for testing forecast accuracy across three classes of examples. Section \ref{conclusion} summarises the paper and proofs are assembled in the Appendix.

\section{Preliminaries}\label{sec:acc}

Denote by $F_{\gamma}:\mathsf{Y}\times\Gamma\mapsto C[0,1]$ a probability distribution on $(\mathsf{Y},\mathscr{B})$ indexed by a parameter $\gamma\in\Gamma\subseteq\mathbb{R}^p$ and lying in
a model family $\mathcal{M}(\Gamma):=\{F_{\gamma}:\gamma\in\Gamma\}$. We envisage forecasting settings where the practitioner entertains a collection of different possible statistical models that can describe, with
varying degrees of accuracy, the evolution of the stochastic process. We suppose that each model is specified using a (semi-) parametric family that depends on unknown parameters.

Each of the models is combined with transformations $T_{\upsilon}$ that depend on tuning parameters $\upsilon\in\Upsilon\subseteq\mathbb{R}^q$. The transformations characterise data manipulations and other techniques used in the construction of the forecasts (the choice of window width in a rolling window for example). Given a family of technical transformations, $\mathcal{T}(\Upsilon):=\{T_{\upsilon}:\upsilon\in\Upsilon\}$, and a model family $\mathcal{M}(\Gamma)$, the class of probability distributions used for prediction is the composition $Q_\theta:\mathsf{Y}\times\Theta\mapsto C[0,1]$ where (with a slight abuse of notation) $M_\gamma\circ T_\upsilon:(Y_1,\cdots,Y_{T})\mapsto Q_\theta$. We denote by $\mathcal{Q}(\Theta)$ the class of forecast methods $\{Q_\theta:\theta\in\Theta=\Gamma\times\Upsilon\}$. Generally, the model parameters of $Q_\theta$ are unknown and must be estimated, and the tuning parameters are assigned by the practitioner using data directed selection criteria or ``rules of thumb".

For any $1\le n\le T$, let $\Omega_{n}\supseteq\mathscr{B}_n$ denote the information available to the forecaster at time $n$. The forecasters prediction is then given by the functional $\phi[Q_{\theta}(Y_{T+h}|\Omega_T)]$, where
$Q_{\theta}(Y_{T+h}|\Omega_T)$ denotes the distribution of the random variable $Y_{T+h}$ derived from the forecasting method conditional on the information $\Omega_T$. Throughout, we consider that the forecaster wishes to obtain `optimal' forecasts in the spirit of \cite{Gneiting2007}, \cite{Gneiting2011}, and \cite{Martin2022}, meaning that the forecasting method the forecaster chooses to employ out-of-sample is produced by targeting a scoring rule that measures precisely the feature or features of $Y_{T+h}$ that are of interest.

\subsection{Scoring Rules}

Recall that $\mathcal{Q}$ is a class of distributions on $\mathsf{Y}$, and a scoring rule is a measurable function $S(f; y)$ that maps a forecast $f$ and an observation $y$ to a numerical value in $[0,\infty)$ that measures the discrepancy of the forecast, with smaller scores indicating a superior forecast. The functional $\phi:\mathcal{Q}\mapsto\Phi$ maps a distribution $Q\in\mathcal{Q}$ to a
subset $\phi[Q]$ of the feature space $\Phi$, $Q\mapsto\phi[Q]\subseteq \Phi$. A scoring function $S(\cdot,\cdot)$ is said to be $\mathcal{Q}$ consistent or proper for a functional $\phi[\cdot]$ if
\begin{flalign*}
\mathbb{E}_Q\left\{S\left(\phi[Q],Y\right)\right\}\leq \mathbb{E}_Q\left[S(\phi[P],Y)\right]\,,
\end{flalign*}
for all $P,Q\in\mathcal{Q}$, where $Y$ is a random variable with distribution $Q$. Here and throughout $\mathbb{E}_{\mathsf{P}}[X]$ denotes the expectation of the random variable $X$ under $\mathsf{P}$, and we assume that any stated expectation exists and is finite. We say that $S(\cdot,\cdot)$ is strictly $\mathcal{Q}$ consistent or proper for $\phi[\cdot]$ if
\begin{equation*}
\mathbb{E}_Q\left\{S\left(\phi[Q],Y\right)\right\}= \mathbb{E}_Q\left[S(\phi[P],Y)\right]
\end{equation*}%
implies that $\phi[P]=\phi[Q]$. If $\phi[\cdot]$ admits a strictly consistent or proper scoring function, then it is
called \textit{elicitable}. Here and throughout, we will assume that the functional of interested is \textit{elicitable}.

Scoring rules can be used to measure the accuracy of a forecast distribution itself (i.e. $\phi[f]=f$) and they can also be used to measure the accuracy of features of the distribution such as
quantiles, or predictive intervals (see, \citealp{Gneiting2007}). \citep{Gneiting2007a} , \cite{Gneiting2011a}, \cite{Gneiting2011}. Some functionals are not elicitable however, at least not on their own, see \citet{brehmer:gneiting:2021} and \citealp{Fissler2016}. See also \cite{Patton2020} for a discussion of the sensitivity of consistent scoring rules to model misspecification, parameter estimation error, and nonnested information sets.

\subsection{Parameter Estimation}

\label{subsubsec:estforecomb}

Given a realisation of $n$ observations $\{y_t:1\le t\le n\}$, $n\le T-h$, the optimal functional predictive can generally be produced in two possible ways;
\begin{itemize}
  \item[i.] by searching for the most accurate predictive distribution $Q_\theta\in\mathcal{Q}$ under a given loss function $L:\mathcal{Q}\times\mathsf{Y}\mapsto \mathbb{R}$ and substituting into $\phi$ to give $\phi[Q_{\widehat{\theta}_n}(Y_{T+h}|\Omega_T)]$ where
  \begin{equation*}
\widehat\theta_n=\operatornamewithlimits{argmin\,}_{\theta\in\Theta}%
\sum_{t=1}^{n}w_tL[Q_{\theta}(Y_{t+h}|\Omega_t),y_{t+h}]\,,  \label{eq:1step}
\end{equation*}%
or similarly
  \item[ii.] by searching for the most accurate functional predictive $\phi[Q_\theta]$, $Q_\theta\in\mathcal{Q}$, under a given scoring rule and evaluating $\phi[Q_{\widetilde{\vartheta}_n}(Y_{T+h}|\Omega_T)]$ directly where
  \begin{equation*}
\widetilde{\vartheta}_n=\operatornamewithlimits{argmin\,}_{\theta\in\Theta}%
\sum_{t=1}^{n}w_tS[\phi[Q_{\vartheta}(Y_{t+h}|\Omega_t)],y_{t+h}]\,.
\end{equation*}
\end{itemize}
Here $w_t$, $t \in \mathbb{N}$, is a $\mathscr{B}_t$-measurable weighting sequence that is known at the time of forecasting. For example, if $w_t=1$, $t=n-\omega,\ldots,n$, $w_t=0$, $t<n-\omega$, $\omega<n$, then the estimate is based on a data window of width $\omega+1$, whereas $w_t=\lambda^{n-t}$, $t=1,\ldots,n$, $0<\lambda<1$, corresponds to a geometrically weighted expanding window. Practitioners  might also be interested in whether a particular forecast method is more accurate than another given that certain conditions apply. Choosing weights such that $w_t = \rho_t$ if the conditions hold and $w_t = 1-\rho_t$ otherwise, $0<\rho_t<1$, might addresses this question. The requirement that the weights be $\mathscr{B}_t$-measurable reflects that knowledge that one method is preferable to another under given conditions is only useful if the conditions are known ex-anti.

Given a particular forecast method class $\{Q_\theta:\theta\in\Theta=\Gamma\times\Upsilon\}$, the result of either i. or ii. is an optimal forecast distribution, where optimality is assessed with respect to the chosen loss function or scoring rule. We have used the notation $\widehat{\theta}_n$ and $\widetilde{\vartheta}_n$ for the parameter estimator in the two cases since the estimates and the optimal forecasts will not coincide in general even if the forecast model is the same. Any differences between the two will be due to differences in the estimation technique and the forecasting method as apposed to the forecasting model, hence the focus in our analysis on the former rather than the latter.

In the sequential statistical analysis that follows the information used for estimation is increased from $\Omega_n$ to $\Omega_{n+1}$, and according to the value obtained by a scoring rule, $\widehat{\theta}_n$ (respectively $\widetilde{\vartheta}_n$), which are based on $n$ observations, will need to be updated to $\widehat{\theta}_{n+1}$ (respectively $\widetilde{\vartheta}_{n+1}$).\footnote{The use of parameter updates corresponds to the approach adopted in \cite{West1996}. } This can be done recursively using a Newton-Raphson type algorithm. Consider the case where $\widehat{\theta}_n$ is obtained using a geometrically weighted expanding window with tuning parameter $\lambda$. Let $\ell_t(\theta)=L[Q_{\theta}(Y_{t+h}|\Omega_t),y_{t+h}]$, or $\ell_t(\theta)=S[\phi[Q_{\theta}(Y_{t+h}|\Omega_t)],y_{t+h}]$, and suppose that $\ell_t(\theta)$ is twice continuously differentiable, $t=1,\ldots,n+1$. Set $\mathcal{L}_n(\theta)=\sum_{t=1}^{n}\lambda^{n-t}\ell_t(\theta)$. Since by definition of $\widehat{\theta}_n$ we have $\partial\mathcal{L}_{n}(\widehat{\theta}_{n})/\partial\theta=0$, it follows from the Taylor-Young formula that
\begin{align*}
\frac{\partial\mathcal{L}_{n+1}(\widehat{\theta}_{n+1})}{\partial\theta}=&\frac{\partial\mathcal{L}_{n+1}(\widehat{\theta}_n)}{\partial\theta}+\frac{\partial^2\mathcal{L}_{n+1}(\widehat{\theta}_n)}{\partial\theta\partial\theta'}(\widehat{\theta}_{n+1}-\widehat{\theta}_n)
+o(\|\widehat{\theta}_{n+1}-\widehat{\theta}_n\|)\\
 =&\frac{\partial \ell_{n+1}(\widehat{\theta}_n)}{\partial\theta}+H_{n+1}(\widehat{\theta}_n)(\widehat{\theta}_{n+1}-\widehat{\theta}_n)+o(\|\widehat{\theta}_{n+1}-\widehat{\theta}_n\|)\\
 =&0
\end{align*}
where $H_{n+1}(\theta)=\partial^2\mathcal{L}_{n+1}(\theta)/\partial\theta\partial\theta'$. If we assume that for $n$ sufficiently large $\Pr(\|\widehat{\theta}_{n+1}-\widehat{\theta}_n\|<\epsilon)\geq 1-\epsilon$, $\epsilon>0$, and $H_{n+1}(\theta)$ is positive definite at $\widehat{\theta}_n$, which will be the case if $\mathcal{L}_{n+1}(\theta)$ is strictly convex in a neighbourhood of $\widehat{\theta}_{n+1}$, we obtain the stochastic approximation
$$
\widehat{\theta}_{n+1}\approx\widehat{\theta}_n-[H_{n+1}(\widehat{\theta}_n)]^{-1}\partial \ell_{n+1}(\widehat{\theta}_n)/\partial\theta\,.
$$
Starting at an initial value $\widehat{\theta}_{n_0}=\operatornamewithlimits{argmin\,}_{\theta\in\Theta}\mathcal{L}_{n_0}(\theta)$ the previous approximation can be used to obtain $\widehat{\theta}_{n+1}$ recursively for $n\geq n_0$ together with $H_{n+1}(\widehat{\theta}_n)=\lambda H_{n}(\widehat{\theta}_n)+\partial^2\ell_{n+1}(\widehat{\theta}_n)/\partial\theta\partial\theta'$. If $H_{n+1}(\theta)$ is not positive definite it can be replaced with $H_{n+1}(\theta)+\rho_n I$ where $\rho_n$ is chosen at each update, giving an interpolation of Newton-Raphson with the steepest descent method, or by replacing $H_{n+1}(\theta)$ with a matrix that satisfies the so called Newton condition to yield a version akin to a Broyden-Fletcher-Goldfarb-Shanno algorithm. \citep[We refer the reader interested in numerical implementation to][for more detailed particulars]{nocedal:wright:2006}.

\section{Sequential Scoring Rule Evaluation and Termination}\label{SSRE}

\subsection{SSRE Evaluation}\label{essre}
Let $\phi[Q_{\widehat{\theta}_n}(Y_{n+\tau}|\Omega_n)]$, $Q\in\mathcal{Q}$ and $\phi[P_{\widetilde{\vartheta}_n}(Y_{n+\tau}|\Omega_n)]$, $P\in\mathcal{P}$, denote two alternative forecast methods used to forecast the functional $\phi$, which for ease of reference we designate as method Q and method P. Set
$$
R_n(Q,P)=\frac{e^{S(\phi[Q_{\widehat{\theta}_n}(Y_{n+\tau}|\Omega_n)])}}{e^{S(\phi[P_{\widetilde{\vartheta}_n}(Y_{n+\tau}|\Omega_n)])}}\,.
$$
where, without loss of generality, we will often consider $\tau=1$.
Since the scoring rule is negatively orientated, i.e., lower is better, method Q is deemed to be at least as good as method P if $R_n(Q,P)\leq 1$, and vise versa. Define $C_n(Q,P)=\prod_{m=1}^nR_m(Q,P)$, and note that $C_n(Q,P)$ can be used to gauge the accumulation of evidence as more observations accumulate.

For given methods Q and P the performance of the SSRE scheme will be governed by the values of two boundary constants $k_l$ and $k_u$, and the true DGP. The influence of the DGP on the behaviour of SSRE will be formulated here in terms of the impact of the true unknown distribution $\mathsf{P}$ on the ratio of scoring rules, as measured by $R_n(Q,P)$. Denote by $\mathbb{P}$ the set of probability measures on $(\mathsf{Y},\mathscr{B})$, and consider a sequence of natural filtrations $\mathscr{B}_1\subseteq\cdots \subseteq\mathscr{B}_n$, with $Y_n$ adapted to $\mathcal{B}_n$ for all $n\ge1$. We will analyse the SSRE under the following hypotheses,
\begin{align*}
\mathcal{H}_Q =& [\mathcal{P}\in \mathbb{P} : \mathbb{E}_{\mathsf{P}}[R_n(Q,P)\mid\mathscr{B}_n]\leq 1\,,  (n\ge1)]\quad\mbox{and} \\
\mathcal{H}_P =&[\mathcal{P}\in \mathbb{P} : \mathbb{E}_{\mathsf{P}}[R_n(Q,P)\mid\mathscr{B}_n]\geq 1\,,  (n\ge1)]\,.
\end{align*}

The hypotheses state that given the information available at the time of forecasting, forecast method Q is expected to be at least as good as method P under scoring rule $S$, or vice versa.\footnote{\citet{zhu:timmermann:2020} have suggested that the null hypothesis of \textit{no inferior forecast performance} ($\mathcal{H}_Q\cap \mathcal{H}_P$) used in the forecast tests referenced \textit{op. cit.} is unlikely to ever be satisfied in realistic settings.} Indeed, for each $n$, the hypothesis $\mathcal{H}_Q$ implies that
$$
1\geq \mathbb{E}_{\mathbb{P}}[R_n(Q,P)\mid\mathscr{B}_n]\geq \exp\{\mathbb{E}_{\mathbb{P}}[\log(R_n(Q,P))\mid\mathscr{B}_n]\},
$$
where the lower bound follows by Jensen's inequality, so that taking logarithms yields the usual null hypothesis of mean forecast dominance commonly employed in the literature
\begin{equation}\label{usualnull}
0\geq \mathbb{E}_{\mathbb{P}}[\log(R_n(Q,P))\mid\mathscr{B}_n]=\mathbb{E}_{\mathbb{P}}[S(\phi[Q])-S(\phi[P])\mid\mathscr{B}_n].
\end{equation}
The hypotheses $\mathcal{H}_Q$ and $\mathcal{H}_P$ imply (slightly more than) mean dominance at all time points. According to whether $\mathcal{H}_Q$ or $\mathcal{H}_P$ does or does not hold, obviously, evidence in favour or against $\mathcal{H}_Q$ or $\mathcal{H}_P$ will accumulate in $C_n(Q,P)$ as $n$ increases, and thus that a selection of method Q or method P should eventually be made.  Ultimately, our goal is to gauge which method has more empirical support given our observed data.

Given a sequence of observations $\{Y_n: n\geq 1\}$, arriving sequentially through time, we measure the evidence for or against $\mathcal{H}_Q$ using
\begin{flalign*}
C_n(Q,P)=\prod_{m=1}^{n}R_m(Q,P)=&\exp\left\{\sum_{m=1}^{n}\left[S(\phi[Q],Y_{m+\tau})-S(\phi[P],Y_{m+\tau})\right]\right\}\\=&\exp\left\{\sum_{m=1}^{n}D_m[Q,P]\right\}.
\end{flalign*}
If $\mathcal{H}_Q$ is satisfied at each $n\geq1$, then $\mathbb{E}_{\mathbb{P}}[C_n(Q,P)\mid\mathscr{B}_n]\leq 1$, and vice versa under $\mathcal{H}_P$. To gauge the empirical support for  
{$\mathcal{H}_Q$ (or $\mathcal{H}_P$)}, we must consider when $C_n(Q,P)$ is small enough (or large {enough}) to support $\mathcal{H}_Q$ (or $\mathcal{H}_{P}$). To this end, define the events $E^Q_n$, $E^P_n$ and $E^{Q\cap P}_n$, for each $n$, by the inequalities
$$
C_n(Q,P)\leq k_l\,,\quad C_n(Q,P)\geq k_u\quad\mbox{and}\quad k_l<C_n(Q,P)<k_u\,,\quad 0<k_l<1<k_u.
$$
The formulation of $C_n(Q,P)$ ensures that, for any $n\geq1$, $C_n(Q,P)\in\mathbb{R}_+$, so that we can view the constants $k_l$ and $k_u$ as defining regions of plausibility for $\mathcal{H}_Q$ or $\mathcal{H}_P$. In particular, the formulation of $C_n(Q,P)$ and the specification of $\mathcal{H}_Q$ and $\mathcal{H}_P$ is such that the following diagram illustrates how the boundary constants $k_l$ and $k_u$ determine regions where $C_n(Q,P)$ favours either $\mathcal{H}_Q$ or $\mathcal{H}_P$.

\vspace{0.5cm}

\begin{figure}[!h]
	\centering

\begin{tikzpicture}
\centering
\tikzset{
	position label/.style={
		below = 3pt,
		text height = 1.5ex,
		text depth = 1ex
	},
	brace/.style={
		decoration={brace, mirror},
		decorate
	}
}
\draw (0,0) -- (10,0);
\foreach \x in {0,2.3,4,6,10}
\draw (\x cm,3pt) -- (\x cm,-3pt);

\node [position label] (cStart) at (0,0) {$ 0 $};
\node [position label] (cA) at (2.3,0) {$ k_l $};
\node [position label] (cB) at (4,0) {$ 1 $};
\node [position label] (cD) at (6,0) {$ k_u$};
\node [position label] (cEnd) at (10,0) {$\infty$};

\draw [brace] (cStart.south) -- node [position label, pos=0.5] {$C_n(Q,P)$ favours $\mathcal{H}_Q$} (cA.south);
\draw [brace,decoration={raise=4ex}](cD.south) -- node [position label,yshift=-4ex] {$C_n(Q,P)$ favours $\mathcal{H}_P$} (cEnd.south);
\draw [decorate,decoration={brace,mirror,raise=-8ex}](cA.south) -- node [position label,yshift=12ex] {SSRE indifference region} (cD.south);
	
\end{tikzpicture}
\caption{Regions of favorability for $\mathcal{H}_Q$ and $\mathcal{H}_P$ in terms of the SSRE boundary constants $k_l$ and $k_u$.}
\end{figure}
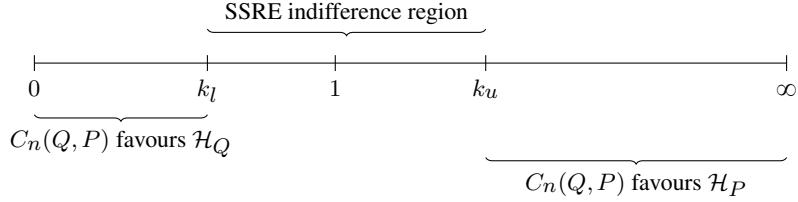

Define a SSRE decision rule by reference to the preference-indifference regions
$$
(y_1,\ldots,y_n)\in\left\{
\begin{array}{ll}
	E^Q_n, & \hbox{at time point $n$ select method Q and cease SSRE;} \\
	E^P_n, & \hbox{at time point $n$ select method P and cease SSRE;} \\
	E^{Q\cap P}_n, & \hbox{at time point $n$ add data point $y_{n+1}$ and continue SSRE.}
\end{array}
\right.
$$
The $2n+1$ events $E^Q_1,\ldots,E^Q_n$, $E^P_1,\ldots,E^P_n$ and $E^{Q\cap P}_n$ form a disjoint partition of the sample space of $Y_1,\ldots,Y_n$, which together with the previous preference-indifference decision rule, define a SSRE scheme for selecting a preferred forecasting method.

In a typical testing setting, one would accept or reject $\mathcal{H}_{Q}$ (or $\mathcal{H}_P$) according to some measure of statistical evidence. However, the nature of $\mathcal{H}_Q$ and $\mathcal{H}_P$, and the sequential nature of data arrival, makes the specification of such criteria cumbersome. In particular, the regions that define the SSRE scheme, which are represented in terms of $k_l$ and $k_u$, govern the performance of the scheme. Before examining the relationship between SSRE and the related literature on E-variables, we will first analyse more closely how the choice of boundary values $k_l$ and $k_u$ influences features of SSRE.
%
%
\subsection{Termination of SSRE}\label{termination}
The SSRE process terminates at the smallest integer $n$ for which the inequality $ k_l< C_n(Q,P)< k_u$ fails to hold. Thus $E^{Q\cap P}_n$ is the event resulting in the outcome $N>n$ where $N$ is the random variable denoting the SSRE stopping time. The probability that the SSRE process terminates equals
$$
  \Pr(N<\infty)= \lim_{n\rightarrow\infty}\Pr(N<n)=1-\lim_{n\rightarrow\infty}\Pr(N\geq n)\,.
$$
Thus, if $\lim_{n\rightarrow\infty}\Pr(N\geq n)>0$ the SSRE process will have a positive probability of continuing indefinitely. The following result shows that, in fact, the SSRE stopping time $N$ is finite with probability one, and the tail of its distribution declines towards zero at an exponential rate, under very weak conditions.
\begin{proposition}\label{prop1}
If under $\mathcal{H}_Q$ or $\mathcal{H}_P$ the distribution of $C_n(Q,P)=\prod_{m=1}^nR_m(Q,P)$ is nondegenerate and $\Pr(E^Q_n)\neq 0$ or $\Pr(E^P_n)\neq 0$, then for all $n\geq n_\epsilon+1$ there exists a $\varrho_\epsilon$ with $0<\varrho_\epsilon<1$ such that $\Pr(N>n)<\exp\{n_\epsilon\log(\varrho_\epsilon)\}$ and the SSRE process will terminate with probability one.
\end{proposition}

If the SSRE process terminates with probability one then the expected value of $N$, the average number of time points required before either forecast method Q or forecast method P is selected, will be of interest. The following result indicates that all moments of $N$ will in fact exist under the same conditions as for Proposition \ref{prop1}.
\begin{corollary}\label{corol1}
Assume that the conditions of Proposition \ref{prop1} hold. Then the moment generating function $\psi_N(s)=\mathbb{E}[\exp(Ns)]$ exists for all $s$ in a neighbourhood of zero.
\end{corollary}

\noindent
Since the moment generating function of $N$ is convergent in a region containing the origin it follows that $\psi_N(s)$ is continuously differentiable at $s=0$ and hence that $E[N^r]$ exists for $r=1,2,\ldots$.

Proposition \ref{prop1} and its corollary generalise to SSRE results known to hold for the standard SPRT \citep[a test of a simple null hypothesis against a simple alternative using the likelihood ratio statistic from independent and identically distributed data,][]{wetherill:1975}. Here however the hypotheses $\mathcal{H}_Q$ and $\mathcal{H}_P$ are highly composite, and simplicity of the hypotheses is replaced by the condition that the statistic $C_n(Q,P)$ that replaces the likelihood ratio is based on a consistent (proper) scoring rule for the feature of interest. More importantly, $\mathcal{H}_Q$ and $\mathcal{H}_P$ do not specify the structure of the stochastic process giving rise to the observed series, which is usually not well enough understood to formulate a suitable stochastic process model for the DGP. The forecaster is therefore at liberty to apply SSRE to any forecasting method and data that they choose.


\subsection{Boundary Values $k_l$ and $k_u$}\label{dbvk}
The probability of selecting method Q at time point $n$, $\Pr(E^Q_n|\mathcal{H}_a)$, or selecting method P, $\Pr(E^P_n|\mathcal{H}_a)$, or of observing an additional value of the series, $\Pr(E^{Q\cap P}_n|\mathcal{H}_a)$, obviously sum to one under each of $\mathcal{H}_a$, $a=Q,P$. Furthermore, we have already noted that the $2n+1$ events $E^Q_1,\ldots,E^Q_n$, $E^P_1,\ldots,E^P_n$ and $E^{Q\cap P}_n$ also form a disjoint partition of the sample space of $Y_1,\ldots,Y_n$, so
\begin{align*}
  \Pr(\bigcup_{i=1}^nE^Q_{i}|\mathcal{H}_a) & = \sum_{i=1}^n\Pr(E^Q_i|\mathcal{H}_a)\,, \\
 \Pr(\bigcup_{i=1}^nE^Q_n|\mathcal{H}_a) & =\sum_{i=1}^n\Pr(E^P_i|\mathcal{H}_a)\quad\mbox{and} \\
  \Pr(E^{Q\cap P}_n|\mathcal{H}_a) & =\Pr(N>n|\mathcal{H}_a)\,,
\end{align*}
are non-negative and also sum to one. Since $\Pr(\bigcup_{i=1}^nE^Q_{i}|\mathcal{H}_a)$ and $\Pr(\bigcup_{i=1}^nE^P_{i}|\mathcal{H}_a)$ are nondecreasing and fall in the interval $[0,1]$ it follows that under $\mathcal{H}_Q$ or $\mathcal{H}_P$ both
$$
\Pr(E^Q_{\infty}|\mathcal{H}_a)=\lim_{n\rightarrow\infty}\Pr(\bigcup_{i=1}^nE^Q_{i}|\mathcal{H}_a)\quad\mbox{and}\quad\Pr(E^P_{\infty}|\mathcal{H}_a)=\lim_{n\rightarrow\infty}\Pr(\bigcup_{i=1}^nE^P_{i}|\mathcal{H}_a)
$$
exist. These limits give the probability of selecting method Q and the probability of selecting method P, respectively. Although these may be of theoretical interest, the forecaster will presumably be far less concerned with probabilities that stretch out into the infinite future than they will be with properties of SSRE schemes over finite time horizons. Moreover, Proposition \ref{prop1} establishes that the SSRE process will terminate with probability one under $\mathcal{H}_Q$ or $\mathcal{H}_P$, i.e. $\Pr(E^{Q\cap P}_{\infty}|\mathcal{H}_a)=\lim_{n\rightarrow\infty}\Pr(N>n|\mathcal{H}_a)=0$, $a=Q,P$.

Consider then the mutually exclusive events
$$
\mathcal{T}^Q_{n}=\{\bigcap_{i=1}^{n-1} E^{Q\cap P}_{i}\}\bigcap E^Q_n,\quad \mathcal{T}^P_{n}=\{\bigcap_{i=1}^{n-1}E^{Q\cap P}_{i}\}\bigcap E^P_n\quad\mbox{and}\quad \mathcal{C}^{Q\cap P}_{n}=\bigcap_{i=1}^{n}E^{Q\cap P}_{i}\,.
$$
These events represent the time paths over the time interval $t=1,\ldots,n$ leading to the outcomes
$$
(y_1,\ldots,y_n)\in\left\{
  \begin{array}{ll}
    \mathcal{T}^Q_{n}, & \hbox{SSRE terminates by select method Q;} \\
    \mathcal{T}^P_{n}, & \hbox{SSRE terminates by select method P;} \\
    \mathcal{C}^{Q\cap P}_{n}, & \hbox{SSRE continues.}
  \end{array}
\right.
$$
It follows that $\beta_Q=\Pr(\mathcal{T}^Q_{n}|\mathcal{H}_P)$ represents the probability of terminating SSRE at time point $n$ by selecting method Q when in fact the expected score is minimised by method P, and that $\beta_P=\Pr(\mathcal{T}^P_{n}|\mathcal{H}_Q)$ represents the probability of terminating SSRE at time point $n$ by selecting method P when in fact the expected score is minimised by method Q.
An obviously desirable property of any SSRE scheme is that $k_l$ and $k_u$ should be set so that the risks of making an incorrect decision on termination are suitably bounded.\footnote{The risk associated with a SSRE incorrect decision is analogous to a hypothesis test type II error, but $\mathcal{H}_Q$ and $\mathcal{H}_P$ do not play the conventional role of null and alternative hypotheses here.}

In order to determine values of $k_l$ and $k_u$ that will ensure that $\beta_Q$ and $\beta_P$ achieve preassigned levels we will use the following lemma. This result generalises a moment generating function (MGF) property due to Wald that he used to approximate the operating characteristic of the standard SPRT \citep[][\S\,A 2.3]{wald:1947}. In the following lemma, and in subsequent integrals with respect to $\mathsf{P}$, the integral is taken over $\mathcal{T}^Q_{n}\cup\mathcal{T}^P_{n}\cup\mathcal{C}^{Q\cap P}_{n}$.
\begin{lemma}\label{lem:mgf}
	Assume that $\mathbb{E}_{\mathsf{P}}[\log(C_n(Q,P))]\neq 0$ and that
	$$
	\mathbb{E}_{\mathsf{P}}(\exp[h\log\{C_n(Q,P)\}])=\int\exp[h\log\{C_n(Q,P)\}]d\mathsf{P}<\infty
	$$
	for all $h\in\mathbb{R}$. Suppose, in addition, there exists an $\epsilon>0$ such that $\Pr(C_n(Q,P)>1+\epsilon)>0$ and $\Pr(C_n(Q,P)<1-\epsilon)>0$. Then there exists a value of $h\neq 0$, $h=h_{\mathsf{P}}$ say, where $h_{\mathsf{P}}$ and $\mathbb{E}_{\PP}[\log(C_n(Q,P)]$ are of opposite sign and $h_{\mathsf{P}}$ is such that $\mathbb{E}_{\PP}[\exp(h_{\mathsf{P}}\log(C_n(Q,P)))]=1$.
\end{lemma}
\begin{proof}
Set $m(h)=\mathbb{E}_{\mathsf{P}}[\exp(hZ)]$ where $Z=\log(C_n(Q,P)$. For $h>0$ we have
$$
m(h)>\exp\{h\log(1+\epsilon)\}\Pr(Z>\log(1+\epsilon))=(1+\epsilon)^h\Pr(C_n(Q,P)>1+\epsilon)\,,
$$
and for $h<0$
$$
m(h)>\exp\{-h\log(1-\epsilon)\}\Pr(Z<\log(1-\epsilon))=(1-\epsilon)^{-h}\Pr(C_n(Q,P)<1-\epsilon)\,,
$$
from which it is evident that $m(h)\rightarrow\infty$ as $|h|\rightarrow\infty$. Denoting the first and second derivatives of $m(h)$ by $m'(h)$ and $m''(h)$ we have $m'(0)=\mathbb{E}_{\mathsf{P}}[Z]=\mu\neq 0$ and $m''(0)=\mathbb{E}_{\mathsf{P}}[Z^2\exp(hZ)]> 0$. It follows that $m(h)$ is a strictly convex function that increases to infinity as $|h|\rightarrow\infty$, and since $m(0)=1$ there exists a value $h_{\mathsf{P}}\neq 0$ such that $m(h_{\mathsf{P}})=1$. Rolle's theorem now implies that $m(h)$ has a unique minimum at a value $h_0$, $|h_0|<|h_{\mathsf{P}}|$, with $m(h_0)< 1$. If $\mu<0$ then $m(h)$ is decreasing at $h=0$, which implies that $h_{\mathsf{P}}$ is positive. Conversely, $m(h)$ is increasing at $h=0$ if $\mu>0$ and $h_{\mathsf{P}}$ must be negative.
\end{proof}
The proof of Lemma \ref{lem:mgf} is modelled on that of Wald \citep[Lemma A.1,][\S\,A 2.1]{wald:1947} and is presented here for the sake of completeness and to facilitate translation to SSRE schemes.
%
%

Now, consider the evaluation of $\Pr(\mathcal{T}^Q_{n}|\mathcal{H}_Q)$, and let $\mathds{Q}$ denote the measure with Radon-Nikodym derivative with respect to $\mathsf{P}$ given by
\begin{equation}\label{eq:rddiv}
\frac{d\mathds{Q}}{d\mathsf{P}}=\exp[h_{\mathsf{P}}\log\{C_n(Q,P)\}]=\exp\left[h_{\mathsf{P}}\left\{\sum_{i=1}^{n}D_i(Q,P)\right\}\right]\,.	
\end{equation}
By definition $\mathds{Q}$ defines a probability distribution on $(\times_{m=1}^n\mathsf{Y}_m,\mathscr{B}_n)$ that is absolutely continuous with respect to $\mathsf{P}$ \citep[][\S\, 31\,\&\,32]{halmos:1950}. From the definition of the SSRE scheme we have that method Q will be selected at time point $n$ if $C_n(Q,P)\leq k_l$. Hence we have that
$$
\mathbf{1}\{\mathcal{T}^Q_{n}\}=\left(\prod_{i=1}^{n-1}\mathbf{1}\{\mathcal{C}^{Q\cap P}_{i}\}\right)\mathbf{1}\{C_n(Q,P)\leq k_l\}\,,
$$
where $\mathbf{1}\{E\}$ denotes the indicator function for the event $E$, and from \eqref{usualnull} and Lemma \ref{lem:mgf} it follows that
\begin{align*}
\mathds{Q}(\mathcal{T}^Q_{n})=\int\mathbf{1}\{\mathcal{T}^Q_{n}\}d\mathds{Q}= & \int\mathbf{1}\{\mathcal{T}^Q_{n}\}\exp\{h_{\mathsf{P}}\log(C_n(Q,P))\}d\mathsf{P} \\
   \leq & \int k_l^{h_{\mathsf{P}}}\mathbf{1}\{\mathcal{T}^Q_{n}\}d\mathsf{P}=k_l^{h_{\mathsf{P}}}\Pr(\mathcal{T}^Q_{n}|\mathcal{H}_Q)\,,
\end{align*}
and $\mathds{Q}(\mathcal{T}^Q_{n})=k_l^{h_{\mathsf{P}}+\epsilon_Q}\Pr(\mathcal{T}^Q_{n}|\mathcal{H}_Q)$ where
$$
\epsilon_Q=\frac{\log\mathds{Q}(\mathcal{T}^Q_{n})-\log k_l^{h_{\mathsf{P}}}\Pr(\mathcal{T}^Q_{n}|\mathcal{H}_Q)}{\log k_l}\geq 0\,.
$$

A parallel argument shows that
$$
\mathds{Q}(\mathcal{T}^P_{n})\left\{
  \begin{array}{ll}
    \geq k_u^{h_{\mathsf{P}}}\Pr(\mathcal{T}^P_{n}|\mathcal{H}_P)&  \\
    =k_u^{h_{\mathsf{P}}+\epsilon_P}\Pr(\mathcal{T}^P_{n}|\mathcal{H}_P),&\epsilon_P=\frac{\log\mathds{Q}(\mathcal{T}^P_{n})-\log k_u^{h_{\mathsf{P}}}\Pr(\mathcal{T}^P_{n}|\mathcal{H}_P)}{\log k_u}\geq 0\,.
  \end{array}
\right.
$$

We have already seen that $\Pr(\mathcal{C}^{Q\cap P}_{n}|\mathcal{H}_a)=\int \mathbf{1}\{\mathcal{C}^{Q\cap P}_{n}\}d\mathsf{P}\rightarrow 0$ as $n\rightarrow\infty$, from which it follows that $\lim_{n\rightarrow\infty}\mathds{Q}(\mathcal{C}^{Q\cap P}_{n})=0$. Similarly, if termination occurs at time point $n$ then $\mathsf{P}(\mathcal{C}^{Q\cap P}_{n})=\mathsf{P}(\emptyset)=0$, so $\mathds{Q}(\mathcal{C}^{Q\cap P}_{n})=0$, and
\begin{equation}\label{probsum}
 \mathds{Q}(\mathcal{T}^Q_{n})+\mathds{Q}(\mathcal{T}^P_{n})=1=\Pr(\mathcal{T}^Q_{n}|\mathcal{H}_a)+\Pr(\mathcal{T}^P_{n}|\mathcal{H}_a)\,.
\end{equation}

By definition, $\beta_Q=\Pr(\mathcal{T}^Q_{n}|\mathcal{H}_P)$, and hence from \eqref{probsum} we have $1-\beta_Q=\Pr(\mathcal{T}^P_{n}|\mathcal{H}_P)$, with corresponding definitions for $\beta_P$ and $1-\beta_P$ obtained by interchanging the roles of Q and $P$. Applying the equality in \eqref{probsum} in conjunction with the relationship between $\Pr(\mathcal{T}^Q_{n}|\mathcal{H}_Q)$ and $\Pr(\mathcal{T}^P_{n}|\mathcal{H}_P)$, and $\mathds{Q}(\mathcal{T}^Q_{n})$ and $\mathds{Q}(\mathcal{T}^P_{n})$ on one hand, and $\beta_Q$ and $\beta_P$ on the other, we find, after some somewhat tedious algebra, that
\begin{equation}\label{errorrate}
\beta_Q=\frac{k_l^{h_{\mathsf{P}}+\epsilon_Q}(k_u^{h_{\mathsf{P}}+\epsilon_P}-1)}{k_u^{h_{\mathsf{P}}+\epsilon_P}-k_l^{h_{\mathsf{P}}+\epsilon_Q}}\quad\mbox{and}
\quad\beta_P=\frac{1-k_l^{h_{\mathsf{P}}+\epsilon_Q}}{k_u^{h_{\mathsf{P}}+\epsilon_P}-k_l^{h_{\mathsf{P}}+\epsilon_Q}} \,,
\end{equation}
from which the following results follow;
$$
\frac{\beta_Q}{1-\beta_P}=k_{lQ}^{h_{\mathsf{P}+\epsilon_Q}}<1<k_{uQ}^{h_{\mathsf{P}+\epsilon_P}}=\frac{1-\beta_Q}{\beta_P}\,,
$$
under $\mathcal{H}_Q$, and under $\mathcal{H}_P$
$$
\frac{\beta_P}{1-\beta_Q}=k_{lP}^{(h_{\mathsf{P}+\epsilon_P})}<1<k_{uP}^{(h_{\mathsf{P}+\epsilon_Q})}=\frac{1-\beta_P}{\beta_Q}\,.
$$
This leads us to the following boundary value result.
\begin{proposition}\label{boundary}
In order to achieve preassigned levels of risk $\beta_Q$ and $\beta_P$ on termination, the boundary constants $k_l$ and $k_u$ should be set equal to
$$
k_{l}=\left(\frac{\beta_Q}{1-\beta_P}\right)^{1/(h_{\mathsf{P}}+\epsilon_Q)}\quad\mbox{and}\quad k_{u}=\left(\frac{1-\beta_Q}{\beta_P}\right)^{1/(h_{\mathsf{P}}+\epsilon_P)}\,.
$$
\end{proposition}
The boundary values obtained in Proposition \ref{boundary} correspond to the boundary inequalities of standard SPRT \citep[\textit{cf}.][Sections 2.3 and 2.3]{wetherill:1975}. Although they are relevant, the use of Proposition \ref{boundary} to implement SSRE is infeasible as it stands because neither $h_{\mathsf{P}}$ nor the perturbations $\epsilon_Q$ and $\epsilon_P$ will be known.
We will delay a consideration of how the parameters $k_l$ and $k_u$ can be determined in practice until after we have outlined the consequences of not terminating SSRE within a given time frame.
\subsection{Implementation and Truncation}\label{truncation}
Suppose that a forecaster has two forecast methods $\phi[Q_{\widehat{\theta}_n}(Y_{n+\tau}|\Omega_n)]$ and $\phi[P_{\widetilde{\vartheta}_n}(Y_{n+\tau}|\Omega_n)]$ that they contemplate using to forecast the functional $\phi$, and that they wish to select method Q or method P before making a prediction of the functional $\phi[Y_{T+\tau}]$ at time $T$. Assume that they collect a sample of $T$ values of $Y$, $y_1,\ldots,y_{T}$, and partition the sample into the first $R\ll T$ values of `training data' used to evaluate $\widehat{\theta}_0$ and $\widetilde{\vartheta}_0$, and out-of-sample `test data' observations $y_{R+n}$ for $n = 1,\ldots,T-R$ used to evaluate $C_n(Q,P)=\prod_{m=1}^nR_m(Q,P)$ and implement SSRE using boundary values $k_{l}$ and $k_{u}$. If SSRE terminates for $n\leq T-R$ then they will use the method selected to produce their forecast. If the event $\mathcal{C}^{Q\cap P}_{T-R}$ occurs, however, an adjustment to the decision process must be applied, assuming that further observations at time points $T+\tau$, $\tau>0$, are not available at time $T$.

A not unreasonable rule to follow is to divide the event $\mathcal{C}^{Q\cap P}_{n}$ for $n=T-R$ into the union of
$$
\overline{\mathcal{T}}^Q_{n}=\mathcal{C}^{Q\cap P}_{n}\bigcap \overline{E}^Q_n\quad\mbox{and}\quad\overline{\mathcal{T}}^P_{n}=\mathcal{C}^{Q\cap P}_{n}\bigcap \overline{E}^P_n
$$
where $\overline{E}^Q_n=\{k_{l}<C_n(Q,P)<1\}$ and $\overline{E}^P_n=\{1<C_n(Q,P)<k_{u}\}$, and select method Q if $\overline{\mathcal{T}}^Q_{n}$ occurs and select method P if $\overline{\mathcal{T}}^P_{n}$ occurs. This adjustment is equivalent to truncating SSRE at the last data point and replacing $\mathcal{C}^{Q\cap P}_{n}$ with the union of
$$
\{\bigcap_{i=1}^{n-1} E^{Q\cap P}_{i}\}\bigcap \overline{E}^Q_n\quad\mbox{and}\quad\{\bigcap_{i=1}^{n-1}E^{Q\cap P}_{i}\}\bigcap \overline{E}^P_n\,.
$$
Clearly, $\overline{\mathcal{T}}^Q_{n}\bigcap\overline{\mathcal{T}}^P_{n}=\emptyset$, and the event $\overline{\mathcal{T}}^Q_{n}\bigcup\overline{\mathcal{T}}^P_{n}$ only occurs if the SSRE process fails to terminate on or before $n=T-R$.

It is not surprising to find, via an argument that parallels the development preceeding Proposition \ref{boundary},
that the adjustment brought about by the necessity to invoke the truncation decision rule can result in not insignificant increases in error rates. Practitioners may therefore wish to avoid having to use the truncation rule. One way of doing this is to divide the sample of $T$ values of $Y$ in such a way that implementation of SSRE has a high chance of terminating for $n\leq T-R$. From the corollary to Proposition \ref{prop1} we know that the moments of $N$, the termination point of the SSRE, exist, and from Markov's inequality we have that
$$
\Pr\left(N< n\right)>1-\frac{\mathbb{E}_{\mathcal{P}}[N]}{n}\,.
$$
By assigning a value to $\mathbb{E}_{\mathcal{P}}[N]$ this bound can be used to determine the value of $n$ that will give the minimum sample size required to achieve a prescribed probability. Thereby a guide to $\Pr(N>n)$ can be established and a suitable choice for $R$, the length of the training data, made.

\section{Relationship to e-variables, and generalized e-variables}\label{ssree}
For SSRE $\mathcal{H}_Q$ and $\mathcal{H}_P$ do not play the role of null and alternative hypotheses as in conventional, i.e., Neyman-Person type, significance tests. Neither method Q nor method P takes precedence over the other, suggesting that the objective of SSRE is to balance the strength of evidence in favour of one method against the strength of evidence in favour of the other. An assignment in which $\Pr(\mathcal{T}^Q_{n}|\mathcal{H}_P)=\Pr(\mathcal{T}^P_{n}|\mathcal{H}_Q)$ is equivalent to employing what is termed in \citet{poskitt:sengarapillai:2010} a 'balanced test'. Our aim therefore is to determine boundary values $k_l$ and $k_u$ that yield error probabilities that reproduce desired error rates whilst equating method Q error rates with those of method $P$.
\subsection{SSRE and Generalized Universal e-values}\label{sec:gen_e}
To build our SSRE balanced testing approach, we begin by noting that Lemma \ref{lem:mgf} gives an interesting result on the behavior of the MGF of the statistic $\Delta_n(P,Q)=\sum_{m=1}^{n}D_m(Q,P)$.
Rewriting $\EP[C_n(Q,P)^{h}]$ as $\mathbb{E}_{\mathsf{P}}[\exp(h\log(C_n(Q,P)))]$, it follows by Lemma \ref{lem:mgf} that $h_{\mathsf{P}}$ can be represented as
$$
h_{\mathsf{P}}:=-\omega_{\mathsf{P}}\cdot\mathrm{sign}\left\{\EP[\Delta_n(Q,P)]\right\}\,,\quad\mbox{where}\quad\omega_{\mathsf{P}}\in[0,\infty)\,.
$$
The above representation suggests that the strength of evidence for, or against, $\mathcal{H}_Q$ can be gauged by the magnitude of
$$
\EP[\exp\left\{-\omega\cdot\mathrm{sign}\left\{\EP[\Delta_n(Q,P)]\right\}\Delta_n(Q,P)\right\}]\,,
$$
for some unknown values of $\omega\in[0,\infty)$. A formal statement that clarifies the nature of this relationship is given in the following lemma.
\begin{lemma}\label{lem:gooey}
Under the conditions of Lemma \ref{lem:mgf}, and for any $\omega\in[0,|h_{\mathsf{P}}|]$,
	$$
	\begin{cases}\EP[\exp\{\omega \Delta_n(Q,P)\}]&(\text{under }\mathcal{H}_{Q})\\
		\EP[\exp\{\omega \Delta_n(P,Q)\}]&(\text{under }\mathcal{H}_{P})	
	\end{cases}\bigg\}\le 1\,.
	$$
\end{lemma}
\begin{proof}
Under $\mathcal{H}_Q$, $\EP[\Delta_n(Q,P)]<0$, for all $m$, so that by Lemma \ref{lem:mgf} we have that $h_{\mathsf{P}}=\omega_{\mathsf{P}}>0$. Furthermore, from the convexity properties of the MGF established as part of the proof of Lemma 1 we have that under $\mathcal{H}_Q$
$$
\EP[\exp\left\{-\omega \cdot \mathrm{sign}\left\{\EP[\Delta_n(Q,P)]\right\}\Delta_n(Q,P)\right\}]=\EP[\exp\{\omega\Delta_n(Q,P)\}]\le 1,
$$
for all $\omega=h\in[0,h_{\mathsf{P}}]$. Conversely, $h_{\mathsf{P}}=-\omega_{\mathsf{P}}<0$ under $\mathcal{H}_P$ and the same argument, in reverse, implies that $\EP[\exp\{\omega\Delta_n(P,Q)\}]\le 1$
for all $\omega=-h\in[0,|h_{\mathsf{P}}|]$ under $\mathcal{H}_P$.
\end{proof}

While the construction of the SSRE may seem obtuse, Lemma \ref{lem:gooey} indicates that if $h_{\mathsf{P}}$ were known, $C_n^{(\omega)}(Q,P):=\exp\{\omega\cdot \Delta_n(Q,P)\}$, would be an e-variable  for all $\omega\in[0,|h_{\mathsf{P}}|]$, and subsequently could be used to test $\mathcal{H}_{Q}/\mathcal{H}_{P}$. Given a probability space $(\mathsf{Y},\mathscr{B}_n, \mathcal{P})_{n=1}^{T}$, an \textit{e-variable}, $\mathcal{E}_n(\cdot):\mathsf{Y}\rightarrow[0,\infty]$, for a hypothesis $\mathcal{H}$ is an extended random variable that satisfies, for all $Q\in\mathcal{H}$, $\mathbb{E}_{Y\sim Q}[\mathcal{E}_n]\le 1$, $n\le T$. An \textit{e-process} $(M_n)_{n=1}^{T}$, where $T$ can be finite or infinite, is a non-negative stochastic process adapted to the  filtration $\mathscr{B}_n$ and such that for any stopping time $N$ and any $Q\in\mathcal{H}$, $\mathbb{E}_{Y\sim Q}[M_\ell]\le 1$; i.e., $M_t$ is an e-variable for $\mathcal{H}$ for all $t\le \ell$

Although it is possible to use $\{C_n^{(\omega)}(Q,P)\}_{n\ge1}$ in several ways, we believe the most direct way is to use the fact that ${C_n^{(\omega)}(Q,P) }$  can be interpreted as a universal generalized e-value (GUe-value), see \citet{dey2024anytime} and \citet{dey2024multiple}. In particular, the key realization of Lemma \ref{lem:gooey} is that it demonstrates an immediate connection between $C_n^{(\omega)}(Q,P)$ and e-values: Lemma \ref{lem:gooey} implies that, for any $\omega\in[0,|h_{\mathsf{P}}|]$, $C_n^{(\omega)}(Q,P)$ is an e-process under $\mathcal{H}_Q$; while $C_n^{(\omega)}(P,Q)$ is an e-process under $\mathcal{H}_P$. Thus, if $h_{\mathsf{P}}$ were known, we could build tests that deliver reliable control of error probabilities under $\mathcal{H}_Q$ and $\mathcal{H}_{P}$.

The critical property of e-values from the perspective of SSRE is that they can be used to control the overall risk of making an incorrect decision. Using the above definitions, and Ville's inequality, we can show that values of $C_n^{(\omega)}(Q,P)$, respectively $C_n^{(\omega)}(P,Q)$, that bound the size of tail probabilities can be straightforwardly determined.
\begin{lemma}\label{lem:size}
	For any $\pi\in(0,1)$ and $\omega\in[0,|h_{\mathsf{P}}|]$, $\Pr\left[\sup_{n=1,\ldots,T}C^{(\omega)}_n(Q,P)\geq 1/\pi\right]\le \pi$
under $\mathcal{H}_Q$, and under $\mathcal{H}_P$ $\Pr\left[\sup_{n=1,\ldots,T}C^{(\omega)}_n(P,Q)\geq 1/\pi\right]\le \pi$\,.
\end{lemma}
Lemma \ref{lem:size} suggests that SSRE schemes formulated by setting $k_l$ and $k_u$ equal to pre-conceived thresholds will deliver error rates that are suitably controlled regardless of the stopping rule.
Thus, setting $k_l=\{\beta/(1-\beta)\}^{1/\omega}$ and $k_u=\{(1-\beta)/\beta\}^{1/\omega}$ where $0<\beta<\half$ results in error rates $\beta_Q$ and $\beta_P$ that are both simultaneously less than or equal to $\beta/(1-\beta)$.
Such a choice for the boundary constants is likely to produce a conservative procedure, as the only guiding principle in this choice is the forecasters wish to control overall risk, and since the bounds in Lemma \ref{lem:size} are not sharp. Note that the e-value inequalities in Lemma \ref{lem:size} are here controlling what in a conventional Neyman-Person framework would be regarded as type II errors -- $k_l$ controls the probability of terminating by selecting method Q when $\mathcal{H}_P$ holds, and $k_u$ controls the probability of terminating by selecting method P when $\mathcal{H}_Q$ holds.

\subsection{Implementation and False Discovery}
Consider again the scenario introduced in Section \ref{truncation}, where a sample of $T+\tau$ values of $Y$, $y_1,\ldots,y_{T+\tau}$, is partitioned into the first $R\ll T$ `training data' values and out-of-sample `test data' observations $y_{R+n}$ for $n = 1,\ldots,N=T-R$ used to evaluate $C_n(Q,P)$ and implement SSRE using boundary values $k_{l}$ and $k_{u}$. In the event the SSRE process fails to terminate on or before $n=N$, a truncation decision rule was suggested that divides the event $\mathcal{C}^{Q\cap P}_{N}$ into the union of $\overline{E}^Q_N=\{k_{l}<C_N(Q,P)<1\}$ and $\overline{E}^P_N=\{1<C_N(Q,P)<k_{u}\}$, with method Q selected if $\overline{E}^Q_N$ occurs and method P selected if $\overline{E}^P_N$ occurs. Unfortunately, invoking the truncation decision rule can result in significant increases in error rates, and more importantly, it ignores information about the relative merits of method Q and method P contained in $\mathcal{C}^{Q\cap P}_{n}$ for $n=1,\ldots,N-1$.

The critical finding in Section \ref{sec:gen_e} is that due to the behavior of the function $C_n(Q,P)$ elucidated in Lemma \ref{lem:mgf}, by Lemma \ref{lem:gooey} we have that $\{C^{(\omega)}_n(Q,P)\}_{n\ge1}$ is an e-process for $\omega\in[0,|h_{\mathsf{P}}|]$, and thus can be used to control the size of tail probabilities as in Lemma \ref{lem:size}. 
However, even through $\{C^{(\omega)}_n(Q,P)\}_{n\ge1}$ is an e-process, if we were to test $\mathcal{H}_Q$ (or $\mathcal{H}_P$) many times, the resulting procedure would not control the \textit{false discovery rate} (FDR), defined as the expected proportion of falsely rejected hypotheses, i.e. the probability of selecting method P when $\mathcal{H}_Q$ holds (respectively selecting method Q when $\mathcal{H}_P$ holds) in a sequence of tests.

Nevertheless, Wang and Ramdas (2022) have construed a testing algorithm that controls the FDR and only requires that we have access to a collection of e-values $e_1,\dots, e_N$, corresponding to $N$ tests of the hypothesis $\mathcal{H}_Q$, $\mathcal{H}_Q^n$, $n=1,\ldots,N$, where, with an obvious abuse of notation, $\mathcal{H}_Q^n$, where $n\le N$ 
refers to the satisfaction of $\mathcal{H}_Q$ at the point $n\ge1$. The key idea of Wang and Ramdas (2022) is to create a new sequence of e-values through scaling changes:
\begin{itemize}
	\item Let $e_{(1)}\leq e_{(2)}\leq\cdots\leq e_{(N)}$ denote the ordered set of e-variables.
	\item Transform these ordered e-variables into new e-variables $e_{(1)}^\star,\dots,e^\star_{(N)}$ according to $e^\star_{(j)}=(j/N)e_{(j)}$, $j=1,\ldots,N$.
\end{itemize}
 \citet{wang:ramdas:2022} prove that the e-process $\{e_{(n)}^\star:1\leq n\le N\}$ controls the FDR under arbitrary dependence between the original e-values.

Hence, since we know that $C_n^{(\omega)}(Q,P)$ are e-values, for $\omega\in[0,|h_{\mathsf{P}}|]$, then so are the ordered and scaled e-variables
$$
e_{(j)}^\star(Q,P)=C_{(j)}^{(\omega)}(Q,P)(j/N),\quad 1\le j\le N,
$$
where each $C_{(j)}^{(\omega)}(Q,P)$ is based on $n_j\le n$ observations. Thus, if our goal is to assess the validity for a  collection of null hypothesis $\mathcal{H}_Q^1,\dots,\mathcal{H}_Q^N$, when $N$ is some stopping time, this can be achieved using any of the strategies proposed in Vovk and Wang (2021) for combining e-values. Therefore, we follow the suggestion of Dey, Martin and Williams (2024), and propose to use the average e-variable
$$
\overline{\mathcal{E}}^{(\omega)}_N:=\frac{1}{N} \sum_{j=1}^N\frac{j}{N} \cdot C_{(j)}^{(\omega)}(Q,P)
$$to establish the validity of the collection of hypotheses. This then delivers the first useful result, which shows that the error of falsely rejecting a true $\mathcal{H}_Q^n$ among the $\mathcal{H}_Q^1,\dots,\mathcal{H}_Q^N$ can be controlled.
\begin{theorem}\label{meanevaluesize}
	If all of $\mathcal{H}_Q^{(1)}, \ldots, \mathcal{H}_Q^{(N)}$ are true, then, for any $\pi\in(0,1)$, and $N\ge2$,
	$$
	\operatorname{Pr}_{\mathsf{P}}\left\{\overline{\mathcal{E}}^{(\omega)}_N \ge  \pi^{-1}\right\} \leq \pi
	$$
\end{theorem}
\begin{proof}
	From Markov's inequality,
\begin{flalign*}
	\operatorname{Pr}_{\mathsf{P}}\left\{\overline{\mathcal{E}}^{(\omega)}_N \ge  \pi^{-1}\right\} \leq \EP[\overline{\mathcal{E}}^{(\omega)}_N] \pi=\pi\frac{1}{N}\sum_{j=1}^{N}\frac{j}{N}\EP \left[C_{(j)}^{(\omega)}(Q,P)\right]&={\pi}\frac{N(N+1)}{2N^2}\\&\le \pi,
\end{flalign*}where the second equality comes from the fact that, for each $\omega\in[0,|h_{\mathsf{P}}|]$, by Lemma \ref{lem:gooey} $\EP \left[C_{(j)}^{(\omega)}(Q,P)\right]\le 1$, and the third from the fact that $2N^2> N(N+1)$ for all $N\ge2$.
\end{proof}

Alternatively, if at least one of the collection of hypotheses is invalid, we can show that $\overline{\mathcal{E}}^{(\omega)}_N $ will be larger than $1/\pi$ with probability converging to one. To state such a result, we first require the following intermediate result, the proof of which is given in the appendix.

\begin{lemma}\label{lem:intermediate}
Suppose that $|n^{-1}\Delta_n(Q,P)-\EP [n^{-1}\Delta_n(Q,P)]|=o_p(1)$. If $\lim_n\EP [n^{-1}\Delta_n(Q,P)]=\vartheta>0$, then as $n\rightarrow\infty$
$$
\Pr\{C_{n}^{(\omega)}(Q,P)\ge \pi^{-1}\}\ge 1-o(1),\forall\pi\in(0,1),\;\omega\in[0,|h_{\mathsf{P}}|].
$$
\end{lemma}

To obtain our next result, recall that each $C_{(j)}^{(\omega)}(Q,P)$ are calculated using $n_j$ observations. In what follows, we then assume that, for $n_0\le n_j$ for each $n_j$, $j\le N$, and that $n_0\rightarrow\infty $ as $T\rightarrow\infty$.

\begin{theorem}\label{meanevaluepower}
Suppose that $|n_0^{-1}\Delta_{n_0}(Q,P)-\EP [n_0^{-1}\Delta_{n_0}(Q,P)]|=o_p(1)$ as $T\rightarrow\infty$, and that $\mathcal{H}_Q^{j}$ is false for some $j\le N$. Then for each $\omega\in(0,|h_{\mathsf{P}}|)$,  $\overline{\mathcal{E}}^{(\omega)}_N$  satisfies
$$
\Pr\left\{\overline{\mathcal{E}}^{(\omega)}_N \geq1/\pi\right\}\ge 1-o(1).
$$
\end{theorem}

\begin{proof}
	For any $j\le N$,
	$$
\overline{\mathcal{E}}^{(\omega)}_N\geq \frac{j}{N^2} C_{(j)}^{(\omega)}(Q,P)=\frac{j}{N^2}C_{n_j}^{(\omega)}(Q,P).
	$$ Consequently,
	$$
	\Pr\left\{\overline{\mathcal{E}}^{(\omega)}_N\ge1/\pi\right\} \geq \Pr\left(\frac{j}{N^2} C_{n_j}^{(\omega)}(Q,P) \geq1/\pi\right)=\Pr\left\{C_{n_j}^{(\omega)}(Q,P) \geq\left(\frac{j \pi}{N^2}\right)^{-1}\right\}
	$$
	Note that $j/N^2\in(0,1)$ by construction. Hence, $j\pi/N^2\in(0,1)$. Thus, if $\mathcal{H}_Q^j$ is false, by Lemma \ref{lem:intermediate}, we have 	
	$$
	\Pr\left\{\overline{\mathcal{E}}^{(\omega)}_N\ge1/\pi\right\} \geq \Pr\left\{G_{n_j}^{(\omega)}(Q,P) \geq\left(\frac{j \pi}{N^2}\right)^{-1}\right\}\ge 1-o(1).
	$$
\end{proof}

{Theorems \ref{meanevaluesize} and \ref{meanevaluepower} indicate that SSRE schemes can be supplemented by monitoring the average e-variables $\overline{\mathcal{E}}^{(\omega)}_N$ along with $C_n(Q,P)$, for $n=1,\ldots,N$, using the same $k_l$ and $k_u$ critical values. Information contained in $\mathcal{C}^{Q\cap P}_{n}$, $n=1,\ldots,N$, concerning the relative merits of method Q and method P will thereby be exploited and incorporated into the sequential decision making process. By setting $k_l$ and $k_u$ equal to values that will deliver preassigned terminal error rates, as in Section \ref{sec:gen_e}, suitable control of the FDR will also be achieved whilst maintaining the desired error rates, regardless of the stopping rule.}

\section{Empirical Implementation}\label{sec:implement}

\subsection{Selection of $\omega$}
Implementation of the SSRE procedure based on GUe-variables requires selecting a value of $\omega$, with the validity of our theoretical results requiring that this choice of $\omega$ be smaller than the unknown $|h_{\mathsf{P}}|$.
While such { a constraint} might appear to be a hindrance to implementation, this issue is not as difficult as it first seems. In particular, we know from the proof of Lemma \ref{lem:mgf} that under $\mathcal{H}_Q$, the function $m(\omega)=\E \exp\{-\omega\cdot \text{sign}\left\{\EP\Delta_n(Q,P)\right\}\Delta_n(Q,P)\}$ takes a value of unity at $\omega=0$ and { $\omega_{\mathsf{P}}>0$}. Further, $m(\omega)$ is less than unity for all $\omega\in[0,\omega_{\mathsf{P}}]$. Hence, by Role's theorem we know that $m(\omega)<1$ for all $\omega \in(0,\omega_{\mathsf{P}})$.

Thus, a possible strategy {  is not to select a specific value of $\omega$, but to choose several $\omega\in(0,1]$} and compare the behavior of the statistics
$$
\exp\{-\omega \cdot\text{sign}[\Delta_R(Q,P)]\Delta_R(Q,P)\}
$$across these different choices, where $R$ represents the length of the in-sample period. If all the statistics behave similarly, this is meaningful evidence that we have not inadvertently chosen $\omega> \omega_{\mathsf{P}}$. The behavior of the statistics can be plotted visually across the in-sample set and the resulting figures compared to determine that their behavior is similar.\\

{
Assuming that the values of $\omega$ chosen deliver statistics with similar behavior, each specific value of $\omega$ chosen can be used to construct a different $\overline{\mathcal{E}}_N^{(\omega)}$ and the average version of these can be used in place of a single test statistic $\overline{\mathcal{E}}_N^{(\omega)}$: if $\overline{\mathcal{E}}_N^{(\omega_1)}$ and $\overline{\mathcal{E}}_N^{(\omega_2)}$ are e-processes, then so is their average. Thus, given $S$ values $\omega_1,\ldots,\omega_S$ associated with $S$ similar statistics, and a window across which to implement the SSRE scheme, including $N\le T-R$ distinct periods, the SSRE scheme can be based on a sequence of statistics
$$
\left\{\overline{\mathcal{E}}_{N-j}^{\mathrm{Avg}}:=\frac{1}{S}\sum_{s=1}^S\overline{\mathcal{E}}_{N-j}^{(\omega_s)}\;:\;j=K+T-R,\dots, N+K+T-R \right\},
$$
where $K\ge0$ is chosen by the researcher. For example, in many cases, the researcher may only wish to test differences in forecast accuracy over the final $100$ pseudo out-of-sample observations. In such cases, we take $K$ such that $N=100=T-R-K$, i.e., $K=T-R-100$. A test can then be implemented by comparing each subsequent value of $\overline{\mathcal{E}}_{N-j}^{\text{Avg}}$ with an assigned boundary value, and rejecting the assumed null hypothesis once the boundary is breached. Hence, in what follows we use the average e-variable across three different choices of $\omega$ in our experiments: $\omega\in\{1/4,1/2,1\}$.
}
%
%

The error rates of such a SSRE testing procedure can be controlled by setting { the boundary values} $k_l$ and $k_u$ appropriately. To understand how this can be accomplished, let us simplify the setting by assuming we wish to establish whether a particular model, say model Q, is preferable; of course, nothing prohibits us from entertaining either model as being preferable, but for the purpose of selecting $k_l$ and $k_u$, it is simpler to consider evidence for or against one specific model rather than having to consider values of $k_l$ and $k_u$ that are appropriate for both models. {Such an approach implicitly favors model Q as a benchmark and then seeks to find evidence against the benchmark, i.e. evidence in favour of model P.}

Viewing model Q as our benchmark, we can ask the question of how do our theoretical results shed light on what an appropriate value of $k_l$ and $k_u$ might be. Since the suggested testing approach is based on e-variables, we know that if $\mathcal{H}_Q$ is more accurate then, on average, $C_n(Q,P)$ will be less than unity. We can therefore restructure the SSRE scheme by considering a value of $k_l=1$ as suggesting that the benchmark model Q is favored over model $P$, and then { selecting} a value of $k_u>1$ to ensure we can control the error of favouring model P when in fact benchmark model Q is preferable. This leads us to consider a slightly modified version of the SSRE scheme based on the preference and indifference regions illustrated in Figure \ref{piregions}.
\begin{figure}[!h]
	\centering
	\begin{tikzpicture}
		\centering
		\tikzset{
			position label/.style={
				below = 3pt,
				text height = 1.5ex,
				text depth = 1ex
			},
			brace/.style={
				decoration={brace, mirror},
				decorate
			}
		}
		\draw (0,0) -- (10,0);
		\foreach \x in {0,2.3,6,10}
		\draw (\x cm,3pt) -- (\x cm,-3pt);
		\node [position label] (cStart) at (0,0) {$ 0 $};
		\node [position label] (cA) at (2.3,0) {$ k_l=1$};
		\node [position label] (cD) at (6,0) {$ k_u$};
		\node [position label] (cEnd) at (10,0) {$\infty$};
		
		\draw [brace] (cStart.south) -- node [position label, pos=0.5] {$C_n(Q,P)$ favours $\mathcal{H}_Q$} (cA.south);
		\draw [brace,decoration={raise=4ex}](cD.south) -- node [position label,yshift=-4ex] {$C_n(Q,P)$ favours $\mathcal{H}_P$} (cEnd.south);
		\draw [decorate,decoration={brace,mirror,raise=-8ex}](cA.south) -- node [position label,yshift=12ex] {SSRE indifference region} (cD.south);
		
	\end{tikzpicture}
	\caption{Regions of favorability for $\mathcal{H}_Q$ and $\mathcal{H}_P$ in terms of the modified SSRE boundary constants $k_l$ and $k_u$, with $k_l=1$.}
\label{piregions}
\end{figure}
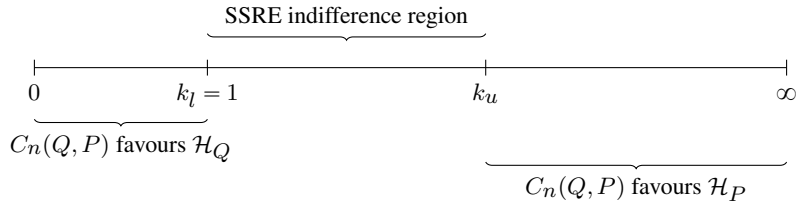

Given this form of the SSRE, we can leverage our theoretical results to understand how to select $k_u$ to control the probability of falsely { favouring P when the benchmark Q holds}. To this end, define the event
$$
\mathcal{T}_{N-j}^P:=\left\{\overline{\mathcal{E}}_{N-j}^{\text{Avg}}\ge k_u\right\},\quad k_u>1,\quad j=0,1\dots,K
$$ which represents termination of SSRE and selection of method P at time $N-j\le N$. We can control the behavior of $\Pr\left(\mathcal{T}_{N}^P\mid { \mathcal{H}^{1}_Q,\dots,\mathcal{H}_Q^N}\right)$ - the probability of selecting method P when {$\mathcal{H}_Q$ holds and} the benchmark method is more accurate - by using Theorem \ref{meanevaluesize}. In particular, fix some $\beta\in(0,1/2)$ and define $\beta_Q=\frac{\beta}{1-\beta}$, to be the error we are willing to tolerate in falsely selecting P when in fact Q is more accurate; i.e.,
$$
\beta_Q=\frac{\beta}{1-\beta}=\Pr\left(\mathcal{T}_{N}^P\mid \mathcal{H}^{1}_Q,\dots,\mathcal{H}_Q^N\right).
$$Then, by Theorem \ref{meanevaluesize}, we know that $\beta_Q\le 1/k_u$, and we can control this error by setting $k_u=1/\beta_Q=(1-\beta)/\beta$.
{Furthermore, by Theorem \ref{meanevaluepower} we have that for each $\omega\in(0,|h_{\mathsf{P}}|)$ the event $\overline{\mathcal{E}}^{(\omega)}_N$ satisfies
$\Pr\left(\overline{\mathcal{E}}^{(\omega)}_N \geq 1/\pi\right)\ge 1-o(1)$ if $\mathcal{H}_Q^{j}$ is false for some $j\le N$, and thus that the probability $\Pr\left(\mathcal{T}_{N}^P\mid \mathcal{H}^{1}_P,\dots,\mathcal{H}_P^N\right)\ge 1-o(1)$ for any $k_u=(1-\beta)/\beta$, $\beta\in(0,1/2)$. Hence we can conclude that when in fact $\mathcal{H}_P$ holds, and method P is more accurate than the benchmark, the probability that we fail to select method P converges to zero.}

{Now observe that the roles of model Q and model P can easily be reversed since the previous constructions and derivations are symmetric in these arguments. Interchanging Q and P and setting method P as the benchmark, a value $k_l$ such that the probability of selecting method Q when $\mathcal{H}_P$ holds, $\Pr\left(\mathcal{T}_{N}^Q\mid {\mathcal{H}^{1}_P,\dots,\mathcal{H}_P^N}\right)$, is controlled, and the probability of selecting method Q when the benchmark does not hold, $\Pr\left(\mathcal{T}_{N}^Q\mid {\mathcal{H}^{1}_Q,\dots,\mathcal{H}_Q^N}\right)$, converges to one, can be determined. In this way a SSRE scheme with boundary constants $k_l$ and $k_u$ that is guaranteed to have desirable statistical properties can be constructed.}

\subsection{Examples}
We now examine the empirical accuracy of our proposed SSRE approach across several common examples in the forecasting literature. In each example, we compare the benchmark model, denoted by $Q$, against a single alternative, denoted by $P$. We compare our SSRE approach  against the most common approach for testing the accuracy of forecasts: a one-sided version of the Diebold and Mariano test, which tests that the benchmark, $Q$, is not inferior to the alternative, $P$. While the two testing approaches are slightly different in their null and alternative hypotheses - we treat a sequence of conditional hypothesis where as the DM test treats a single average hypothesis - such a comparison is warranted as the DM test is far and away the most popular method for testing forecast accuracy.

The key feature of the DM test is that its validity requires the strong assumption that the difference in scores follows a weakly stationary process; which is known to not be satisfied when models are estimated or when working with non-stationary series. In contrast, our testing approach makes no such requirements and builds these features directly into the testing framework. Consequently, it is likely that our approach will deliver accurate results in exactly the situations where standard testing methods will fail.

Across the experiments that follow, we compare the accuracy of {the different models} in terms of  log-scores and  root-mean-squared error. We compare the accuracy of our SSRE scheme and the DM test across $M=1000$ replications from different assumed models, which we present in the subsequent subsections.  Across each example, the SSRE approach uses the testing protocol discussed in the previous section for each experiment; we also set $\beta=1/10$, so that $k_u\approx 11.11$, and conduct the test with this fixed threshold.

For each experiment, the replicated dataset is based on $T=1000$ total observations, which we break in to an in-sample fitting dataset comprised of  $R=T/2=500$ observations, and an out-of-sample test set based on the remaining observations. For simplicity, we consider a fixed in-sample period and do not consider rolling or expanding windows, and fit each model with maximum likelihood. We remark that nothing in the construction of the SSRE prohibits such schemes, and that fixed windows are merely used to avoid model the computational cost of model refitting. For the SSRE approach, we reserve only the last $N=100$ periods to calculate the SSRE test statistics and conduct the sequential testing approach, while the DM test statistic is based on all $T-R=500$ observations in the test set. This feature should positively bias the DM test results due to them having a larger test set, and thus closer to the asymptotic regime. When calculating the DM test statistic we estimate the variance of the loss using HAC based methods, and we reject the null if the corresponding test statistic is greater than $\Phi^{-1}(1-\beta)=1.28$, i.e., if the average difference in the losses is larger than the 10\% quantile of the standard normal, which is a standard rejection region for the one-sided DM test when the losses are negatively oriented.

To gauge the accuracy of our SSRE approach we simulate data from a fixed version of the benchmark Q and investigate the Monte Carlo frequency with which this null is rejected and compare that to the theoretical results, which, by Theorem \ref{meanevaluesize}, under our specific choices should be controlled and less than $\beta/(1-\beta)\approx 11\%$ across all experiments. We then verify the accuracy of Theorem \ref{meanevaluepower} by simulating data under a fixed version of model P and calculating the same rejection frequency, which should diverge as $T$ diverges.

We now briefly present the different simulation experiments and then present the results of these experiments.

\subsubsection{Example 1: Unit Root versus AR(1)}
Consider comparing the accuracy of a benchmark forecasting model, $Q$, given by a unit root process, against the  alternative forecasting model, $P$, given by an autoregressive process of order 1 (AR(1)). Both forecasts are then based on the {specification}:
\begin{equation}\label{eq:model1}
Y_t= \rho Y_{t-1}+\sigma\varepsilon_t,\quad \varepsilon_t\sim N(0,1),\quad t=2,3,\dots,T,	
\end{equation}
where model P enforces $|\rho|<1$, while model Q sets $\rho=1$. In this case, method Q corresponds to the forecasting model $Y_t\mid Y_{t-1}\sim N(\mu+Y_{t-1},\sigma^2)$, while method P corresponds to $Y_t\mid Y_{t-1}\sim N(\mu+\rho Y_{t-1},\sigma^2)$, where $|\rho|<1$. We fix $\sigma=2$ in all experiments, and consider $\rho\in\{1,0.90\}$. The case where $\rho=1$ corresponds to model Q being more accurate than model $P$, while the case $\rho=0.90$ covers the converse. Under the stipulated choice of $k_u$, our theoretical results imply that the empirical rejection frequency should be less than $11\%$ in the first case $(\rho=1)$, while in the second ($\rho=0.90$) our results suggest that the empirical rejection frequency should be large.

\subsubsection{Example 2: AR(2) versus ARMA(2,1)}
In the second example, we generate data from the ARMA(2,1) process:
 \begin{equation}\label{eq:model2}
 	Y_t= \rho_1 Y_{t-1}+\rho_2 Y_{t-2}+\theta\varepsilon_{t-1}+\sigma\varepsilon_t,\quad \varepsilon_t\sim N(0,1),\quad t=2,3,\dots,T.
 \end{equation}We consider two forecasting methods: a benchmark model, $Q$, based on fitting an ARMA(2,1) model, and an alternative, $P$, based on fitting an AR(2). In this example, we use the ARMA(2,1) model to explore situations where neither model should be favored, and so we require that both models deliver forecasts that are of similar accuracy. Following \citet{poskitt:1987}, we know that if data is generated from the  ARMA(2,1) process with parameters $\rho_1=1.4$, $\rho_2=-.6$ and  $\theta_1=0.285$, the resulting forecasts based on the AR(2) model may actually be favored over those of the true ARMA(2,1) model in small-to-moderate samples. Hence, such a setting would deliver a case where the benchmark method Q is indeed more accurate than method P, but where detection of this should be difficult. {Conversely, it was also shown by \citet{poskitt:1987} that if we fit a misspecified AR(2) model to the ARMA(2,1) process the pseudo-true parameter values are given by $\rho_1=1.50$, and $\rho_2=-0.70$}. Thus, if we were to instead generate data from the ARMA(2,1) model under these values of $\rho_1$ and $\rho_2$, while setting $\theta=0$, we would have a situation where the alternative model P should actually be favored over the benchmark $Q$, but where the two models will again produce extremely similar predictions, and so we would not reject the benchmark model Q in such a case.

\subsubsection{Example 3: Forecast Combinations}
As our last example, we compare the case where forecasts are generated from a combination of {time series models using} regression. Forecast combination is one of the most common approaches to empirical forecasting, and testing forecast accuracy in such settings has many well-known issues. In this experiment, the observed data is generated according to
\begin{flalign*}
Y_t=&\omega X_{1t}+(1-\omega)X_{2t}+\sigma\varepsilon_t,\quad \varepsilon_t\sim N(0,1),\quad t=2,3,\dots,T
\\X_{1t}=&0.50\epsilon_{t-1}+0.285\epsilon_{t-2}+\epsilon_t,\quad \epsilon_t\sim N(0,1),\quad t=3,\dots,T\\
X_{2t}=&0.50\nu_{t-1}-2\times0.285\nu_{t-2}+\nu_t,\quad \nu_t\sim N(0,1),\quad t=3,\dots,T.
\end{flalign*}
{The data generating process is designed to mimic a scenario involving alternative leading indicators,} and we do not {attempt} to model the time-series dependence in $X_{1t}$ and $X_{2t}$ { but treat them as exogenous regressors.} Our only goal is to produce point forecasts for $Y_{t+1}$ and so we only consider forecast accuracy in terms of MSE.\footnote{A combination of the models based on the individual model densities will not deliver a predictive density in the correct class, and so we forego analysis using the log score in this example.} Our benchmark method Q is an equally weighted combination of forecasts based on fitting individual regressions of $Y_{t}$ on $X_{1t}$, and $Y_{t}$ on $X_{2t}${. This is tested against an optimally weighted forecast combination based on minimizing the sum of in-sample squared forecast errors, method P}. This particular {alternative} is chosen since in this case the predictive model based on the equally-weighted combination{, model $Q$,} coincides with the {predictive model $P$} when we set $\omega=1/2$. Given this, we consider two scenarios: in the first case we generate data so that {method Q} is approximately correct, which corresponds to setting $\omega=1/2$; in the second, we set $\omega=1/4$ so that the forecast model Q is less accurate than model $P$.

\subsubsection{Results}
We give the results across the different examples in Table \ref{tab:one}, which contains results for which method is more accurate, Q or P, as well as for log-score and MSE, except for Example 3 where we only treat MSE. Across both examples and scores, when model Q is indeed more accurate than model $P$, our test controls the error associated with incorrectly choosing model $P$. Conversely, when model P is more accurate than model $Q$, our test clearly favours model $P$. We remind the reader that in Example 2, the forecasting models are essentially equivalent under the chosen parameterizations and so we should expect that {differentiating between the two models is not feasible}; i.e., under the two designs the rejection rates should be similar, and if they are not, this indicates that the proposed testing method is incorrectly favoring one model over the other.

Comparing our results to those of the DM test, we see that the results of the DM test are more varied. In each scenario we analyze, the size of the DM test is much lower than the nominal level of 10\% used when conducting the test. Such a result suggests that the underlying conditions on which the theoretical behavior of the test is founded are invalid. Generally, when the alternative is correct the DM test accounts for this difference adequately, with two notable exceptions. In Example 1, the SSRE approach has a significantly higher MSE rejection rate than the DM test, and in Example 2 under the log score, the DM test chooses the alternative model about 20\% of the time, when in fact the two models deliver equivalent forecasts under the chosen parameterizations.

These findings support the existing literature that DM-type tests may not be appropriate in complex forecasting comparison settings with non-regular models and estimated parameters. In contrast, we note that our testing approach does not suffer these same issues. Theorem \ref{meanevaluesize} says that if method Q is more accurate, then we can control the probability of falsely selecting method P, but if model P is indeed more accurate than model $Q$, then with probability converging to one we will obtain enough evidence to reject model $Q$. Furthermore, no assumptions about the nature of Q and P are made in our setting.
\begin{table}[!htp]\centering
	\caption{Comparison of SSRE and one sided DM tests. The setting where the benchmark model Q is more accurate is denoted by $\mathcal{H}_Q$, and $\mathcal{H}_P$ denotes when the alternative model P is more accurate. The average rejection probabilities under $\mathcal{H}_Q$ and $\mathcal{H}_P$ for SSRE and DM are presented in columns SSRE and DM respectively.}\label{tab:one}
	\scriptsize
	\begin{tabular}{lrrrrrrrr}\toprule
		& &\multicolumn{2}{c}{Example 1} &\multicolumn{2}{c}{Example 2} &\multicolumn{2}{c}{Example 3} \\\cmidrule{3-8}
		& &SSRE &DM &SSRE &DM &SSRE &DM \\
		\multirow{2}{*}{MSE} &$\mathcal{H}_Q$ &0.010 &0.003 &0.007 &0.001 &0.050 &0.020 \\
		&$\mathcal{H}_P$ & 0.950& 0.810& 0.012&0.003 &0.470 &0.460 \\\cmidrule{1-8}
		\multirow{2}{*}{Log Score} &$\mathcal{H}_Q$ &0.001 &0.001 &0.050 &0.001 &n/a &n/a \\
		&$\mathcal{H}_P$ & 1.000& 1.000&0.001 &0.200 &n/a &n/a \\
		\bottomrule
	\end{tabular}
\end{table}

\section{Conclusion}\label{conclusion}

This paper has considered the theoretical analysis of sequential forecast production and evaluation based on proper scoring rules. We have shown that it is feasible to deliver reliable sequential tests of forecast accuracy under much weaker assumptions than those used in many testing approaches, by leveraging a particular type of test statistic and evaluation framework. The key novelty in this framework is the recognition that the test statistic we have formulated encodes forecast accuracy in a way that it can be represented as a generalized e-variable (\citealp{dey2024anytime}, \citealp{dey2024multiple}) that depends on a particular ``learning rate''. We demonstrate that this learning rate encodes which forecasting model is more accurate, and discuss how it can be set in practice. In leveraging the behavior of e-variables, we show that our sequential testing framework controls the family-wise error rates associated with the testing procedure.

In this manuscript, we have considered the application of this testing procedure to the empirically relevant case where a benchmark forecasting model is compared against some alternative method, and where the null hypothesis is that this benchmark model is at least as accurate as the alternative. However, in theory, the proposed framework should extend beyond this setting. An exciting possibility is the use of this framework for the construction of multi-comparisons across different models, with the ultimate aim being the construction of a sequential model confidence set \citep[see][for a discussion of the concept of a model confidence set]{hansen2011model}. To keep this current paper brief, we have not explored this exciting idea in this manuscript but speculate that it should be feasible to extend our current framework to construct such model confidence sets, while maintaining control on the error rates associated with incorrectly including models in the said set.

\begin{appendix}\label{app:proofs}
\section*{Proofs}
\begin{paragraph}{Proof of Lemma \ref{lem:size}}
Firstly, from the definition of $\mathcal{H}_Q$, for any $n\geq 1$, and $\mathcal{P}\in\mathcal{H}_Q$, $R_n(Q,P)$ is an e-variable: i) it is non-negative; 2) if $\mathcal{P}\in\mathcal{H}_Q$, then $\mathbb{E}_{Y\sim\mathcal{P}}[R_n(Q,P)]\le 1$ for any $n\in\mathbb{N}$. Since each $R_n(Q,P)$ is adapted to $\mathscr{B}_n$, by the tower property of conditional expectations $C_n(Q,P)=\prod_{m=1}^n R_n(Q,P)$, $n\in\mathbb{N}$, is a non-negative super-martingale. In particular, since $R_n(Q,P)\geq0$, and since under $\mathcal{H}_Q$ we have $\mathbb{E}_{Y\sim\mathcal{P}}[R_n(Q,P)\mid \mathscr{B}_{n-1}]\le 1$, it follows that
$$
\mathbb{E}_{Y\sim\mathcal{P}}[C_n(Q,P)]=\mathbb{E}_{Y\sim\mathbb{P}}\left[\mathbb{E}_{Y\sim\mathbb{P}}\left[\prod_{m=1}^n R_m(Q,P)\mid \mathcal{B}_{n-1}\right]\right]=\mathbb{E}_{Y\sim\mathcal{P}}\prod_{m=1}^n\mathbb{E}_{Y\sim\mathcal{P}}[R_m(Q,P)\mid\mathscr{B}_{m-1}],
$$with the convention that $\mathbb{E}_{Y\sim\mathcal{P}}[R_1(Q,P)\mid\mathscr{B}_0]=\mathbb{E}_{Y\sim\mathcal{P}}[R_1(Q,P)]$. Under $\mathcal{H}_{Q}$, $\mathbb{E}_{Y\sim\mathcal{P}}[R_m(Q,P)\mid\mathscr{B}_{m-1}]\le 1$ for each $m\geq1$, we have
$
\mathbb{E}_{Y\sim\mathcal{P}}[C_n(Q,P)]\le 1$. For any $\mathcal{P}\in\mathbb{P}$, the result then follows by Ville's inequality.\hfill \qed
\end{paragraph}
\begin{paragraph}{Proof of Proposition \ref{prop1}:}
Assume that for each $m$ the distribution of $C_m(Q,P)$ is nondegenerate and that the probability of at least one of $E^Q_m$ or $E^P_m$ is non-zero, $m=1,2,\ldots$. Let
$p_m=\Pr(k_l< C_m(Q,P)< k_u))$.  Then $\Pr(N>n)\leq\varrho_n=\prod_{m=1}^np_m$ where $h_m<1$, $m=1,\ldots,n$, and the sequence $\varrho_n$, $n=1,2,\ldots$, converges as $n$ increases since it is monotonically nonincreasing and bounded below by zero.

Suppose that $\lim_{n\rightarrow\infty}\prod_{m=1}^np_m>0$. Cauchy's condition for the convergence of a product implies that $\prod_{m=1}^np_m$ converges if, and only if, for every $\epsilon>0$ there exists an $n_\varepsilon$ such that $|p_{n_\epsilon+1}p_{n_\epsilon+2}\cdots p_{n}-1|<\epsilon$ for all $n>n_\epsilon$. This implies that $p_n\rightarrow 1$ as $n_\epsilon\rightarrow\infty$, contradicting the condition that $p_n<1$. We can therefore conclude that the product $\prod_{m=1}^np_m$ `diverges' and $\lim_{n\rightarrow\infty}\prod_{m=1}^np_m=0$. It follows that $\lim_{n\rightarrow\infty}\Pr(N>n)=0$ and hence that the SSRE process terminates with probability one. Moreover, if we set $\varrho_\epsilon=\sqrt[n_\epsilon]{\prod_{m=1}^{n_\epsilon}p_m}$, then we have that for all $n>n_\epsilon$
$$
\Pr(N>n)\leq\left(\prod_{m=1}^{n_\epsilon}p_m\right)\left(\prod_{m=n_\epsilon+1}^np_m\right)=\varrho_\epsilon^{n_\epsilon}\cdot\left(\prod_{m=n_\epsilon+1}^np_m\right)<\exp\{n_\epsilon\log(\varrho_\epsilon)\}
$$
where $0<\varrho_\epsilon<1$.\hfill \qed
\end{paragraph}
\begin{paragraph}{Proof of Corollory \ref{corol1}:} By definition
\begin{align*}
 \mathbb{E}_{\mathcal{P}}[\exp(Ns)]= & \sum_{n=1}^\infty\exp(ns)\Pr(N=n) \\
  = &\sum_{n=1}^{n_\epsilon}\exp(ns)\Pr(N=n) +\sum_{n=n_\epsilon}^\infty\exp(ns)\{\Pr(N>n+1)-\Pr(N>n)\}\,.
\end{align*}
The second term can be rewritten as $\sum_{k=0}^\infty\exp((n_\epsilon+k)s)\{\Pr(N>n_\epsilon+k+1)-\Pr(N>n_\epsilon+k)\}$. By Proposition \ref{prop1} $\Pr(N>n)< \exp\{n_\epsilon\log(\varrho_\epsilon)\}$ for all $n\geq n_\epsilon+1$, We also have that $\Pr(N>n_\epsilon)\leq \exp\{n_\epsilon\log(\varrho_\epsilon)\}$ and $\Pr(N>n+1)\leq\Pr(N>n)$. So if $\Pr(N>n_\epsilon+k)\leq\exp\{(n_\epsilon+k-1)\log(\varrho_\epsilon)\}$ then $\Pr(N>n_\epsilon+k+1)\leq\exp\{(n_\epsilon+k)\log(\varrho_\epsilon)\}\exp\{-2\log(\varrho_\epsilon)\}$, and hence we find that for $k=0,1,2,\ldots$
$$
|\Pr(N>n_\epsilon+k+1)-\Pr(N>n_\epsilon+k)|\leq\exp\{(n_\epsilon+k-1)\log(\varrho_\epsilon)\}[\exp\{-\log(\varrho_\epsilon)\}+1]\,.
$$
For $s$ such that $s<-\log(\varrho_\epsilon)$ the modulus of the second term in the previous expansion of $\mathbb{E}[\exp(Ns)]$ can therefore be bounded above by
\begin{align*}
  c_\epsilon\sum_{k=0}^\infty\exp((n_\epsilon+k)s)\exp\{(n_\epsilon+k-1)\log(\varrho_\epsilon)\} = & c_\epsilon\exp\{-\log(\varrho_\epsilon)\}\sum_{k=0}^\infty[\exp\{s+\log(\varrho_\epsilon)\}]^{n_\epsilon+k} \\
  = & C_\epsilon\sum_{k=0}^\infty[\exp\{s+\log(\varrho_\epsilon)\}]^k \\
= &\frac{C_\epsilon}{(1-\exp\{s+\log(\varrho_\epsilon)\})}
\end{align*}
where $c_\epsilon=1+\exp\{-\log(\varrho_\epsilon)\}$ and $C_\epsilon=c_\epsilon\exp\{n_{\epsilon}s+(n_{\epsilon}-1)\log(\varrho_\epsilon)\}$. Thus $\mathbb{E}[\exp(Ns)]$ converges to a finite constant for all $s\in(-\infty,-\log(\varrho_\epsilon))$.\hfill \qed
\end{paragraph}

\begin{paragraph}{Proof of Lemma \ref{lem:intermediate}}
	First, let $\Delta_{\mathsf{P}}(Q,P)=\EP \Delta_n(Q,P)$. Decompose $\log C_n^{(\omega)}(Q,P)$ as
	\begin{flalign*}
		\log C_n^{(\omega)}(Q,P)&=\log C_n^{(\omega)}(Q,P)-\omega\EP\Delta_n(Q,P)+\omega\EP\Delta_n(Q,P)\\
		&=\omega \{\Delta_n(Q,P)-\Delta_{\mathsf{P}}(Q,P)\}+\omega\Delta_{\mathsf{P}}(Q,P)\\
		&=\omega \mathcal{D}_n(Q,P)+\omega\Delta_{\mathsf{P}}(Q,P)
	\end{flalign*}	where $\mathcal{D}_n(Q,P)=\{\Delta_n(Q,P)-\Delta_{\mathsf{P}}(Q,P)\}$.
	Now, consider
	\begin{flalign*}
		\Pr\{C_{n}^{(\omega)}(Q,P)< \pi^{-1}\}=&\Pr\left\{\log C_{n}^{(\omega)}(Q,P)< -{\log \pi}\right\}\\
		=&\Pr\left\{\frac{1}{n}\mathcal{D}_n(Q,P)+\frac{1}{n}\Delta_{\mathsf{P}}(Q,P)< -\frac{\log \pi}{n\omega}\right\}
		\\\le &\underbrace{\Pr\left\{\frac{1}{n}\mathcal{D}_n(Q,P)< 0 \right\}}_{\text{Term 1}}+ \underbrace{\Pr\left\{\frac{1}{n}\Delta_{\mathsf{P}}(Q,P)< -\frac{\log \pi}{n\omega} \right\}}_{\text{Term 2}}
	\end{flalign*}where the inequality follows from the union bound. Since $\pi\in(0,1)$, we have $-\log\pi\ge0$. By the hypothesis of the result, $\lim_nn^{-1}\Delta_{\mathsf{P}}(Q,P)=\vartheta>0$, and Term 2 in the above equation converges to zero for all $\omega>0$ as $n\rightarrow\infty$.
	
	It remains to analyze Term 1. To simplify notation, write $\overline{\Delta}_n(Q,P):=n^{-1}\Delta_n(Q,P)$ and $\overline\Delta_{\mathsf{P}}(Q,P):=\EP \overline\Delta_n(Q,P)$. By hypothesis, for any $\varepsilon>0$,
	\begin{flalign}\label{eq:prob}
		\Pr\left\{\overline{\Delta}_n(Q,P)\;:\;\left|\overline{\Delta}_n(Q,P)- \overline\Delta_{\mathsf{P}}(Q,P)\right|\ge \varepsilon\right\}=o(1)\,.	
	\end{flalign}Decompose the event above as
	\begin{flalign*}
		\{|\overline{\Delta}_n(Q,P)-\overline\Delta_{\mathsf{P}}(Q,P)|\ge \varepsilon\}
		&=\left\{\left[\overline{\Delta}_n(Q,P)-\overline\Delta_{\mathsf{P}}(Q,P)\right]\ge \varepsilon\right\}\cup \left\{\left[\overline{\Delta}_n(Q,P)-\overline\Delta_{\mathsf{P}}(Q,P)\right]\le \varepsilon\right\}
		\\&=	\left\{\left[\overline{\Delta}_n(Q,P)-\overline\Delta_{\mathsf{P}}(Q,P)\right]\ge \varepsilon\right\}\cup \left\{\overline{\Delta}_n(Q,P)\le \overline\Delta_{\mathsf{P}}(Q,P)+ \varepsilon\right\}
	\end{flalign*}
	where the second line comes from re-arranging the first set. If $\overline{\Delta}_n\in\left\{\overline{\Delta}_n(Q,P)\le \overline\Delta_{\mathsf{P}}(Q,P)+ \varepsilon\right\}$ for all $\varepsilon>0$, then $	\overline{\Delta}_n\in\left\{\overline{\Delta}_n(Q,P)\le \overline{\Delta}_{\mathsf{P}}(Q,P)\right\}$. Hence, from \eqref{eq:prob} and the above we have
	$$
	\Pr\left\{\frac{1}{n}\mathcal{D}_n(Q,P)< 0 \right\}=\Pr\left\{\overline{\Delta}_n(Q,P)<\overline{\Delta}_{\mathsf{P}}(Q,P)\right\}=o(1)\,,
	$$which proves that
	\begin{flalign*}
		\Pr\{C_{n}^{(\omega)}(Q,P)< \pi^{-1}\}\le o(1)\,.
	\end{flalign*}Taking compliments then yields the stated result.
\end{paragraph}

\end{appendix}


\bibliographystyle{imsart-nameyear}
\bibliography{aos-SSREpapeV9-DavidDon-08-05-25.bbl}

\end{document}